\documentclass[12pt, reqno]{amsart}

\usepackage{amssymb, amsmath, amsthm}
\usepackage[backref]{hyperref}
\usepackage[alphabetic,backrefs,lite]{amsrefs}
\usepackage{amscd}   
\usepackage{fullpage}
\usepackage[all]{xy} 
\usepackage{setspace}
\usepackage{mathtools}

\DeclareFontEncoding{OT2}{}{} 

\usepackage[usenames,dvipsnames]{color}

\newtheorem{lemma}{Lemma}[section]
\newtheorem{theorem}[lemma]{Theorem}
\newtheorem{proposition}[lemma]{Proposition}
\newtheorem{prop}[lemma]{Proposition}

\newtheorem{problem}[lemma]{Problem}
\newtheorem{question}[lemma]{Question}
\newtheorem{claim*}{Claim}

\theoremstyle{definition}
\newtheorem{remark}[lemma]{Remark}

\newcommand{\PP}{{\mathbb P}}
\newcommand{\C}{{\mathbb C}}

\newcommand{\Q}{{\mathbb Q}}

\newcommand{\Z}{{\mathbb Z}}

\newcommand{\eps}{\epsilon}

\newcommand{\calC}{{\mathcal C}}
\newcommand{\calD}{{\mathcal D}}

\newcommand{\calF}{{\mathcal F}}

\newcommand{\calK}{{\mathcal K}}

\newcommand{\frakn}{{\mathfrak n}}

\newcommand{\hideqed}{\renewcommand{\qed}{}}

\DeclareMathOperator{\HH}{H}

\DeclareMathOperator{\rk}{rk}
\DeclareMathOperator{\Char}{char}

\DeclareMathOperator{\Hom}{Hom}

\DeclareMathOperator{\Br}{Br}

\DeclareMathOperator{\SL}{SL}

\DeclareMathOperator{\disc}{disc}

\DeclareMathOperator{\diag}{diag}
\DeclareMathOperator{\Id}{Id}

\newcommand{\isom}{\simeq}

\numberwithin{equation}{section}
\numberwithin{table}{section}

\newcommand{\defi}[1]{\textsf{#1}} 

\title[Kodaira dimension of $\calC_d$]{Kodaira dimension of moduli of special cubic fourfolds}
\author{Sho Tanimoto}
\author{Anthony V\'arilly-Alvarado}

\address{Department of Mathematical Sciences, University of Copenhagen, Universitetspark 5 2100 Copenhagen $\emptyset$ Denmark}
\email{sho@math.ku.dk}
\urladdr{http://shotanimoto.wordpress.com}
\address{Department of Mathematics MS 136, Rice University, 6100 S.\ Main St., Houston, TX 77005, USA}
\email{av15@rice.edu}
\urladdr{http://math.rice.edu/\~{}av15}

\date{}

\begin{document}

	\begin{abstract}
		A special cubic fourfold is a smooth hypersurface of degree three and dimension four that contains a surface not homologous to a complete intersection. Special cubic fourfolds give rise to a countable family of Noether-Lefschetz divisors $\calC_d$ in the moduli space $\calC$ of smooth cubic fourfolds. These divisors are irreducible 19-dimensional varieties birational to certain orthogonal modular varieties. We use the ``low-weight cusp form trick'' of Gritsenko, Hulek, and Sankaran to obtain information about the Kodaira dimension of $\calC_d$. For example, if $d = 6n + 2$, then we show that $\calC_d$ is of general type for $n > 18$, $n \notin \{20,21,25\}$; it has nonnegative Kodaira dimension if $n > 13$ and $n \neq 15$. In combination with prior work of Hassett, Lai, and Nuer, our investigation leaves only $20$ values of $d$ for which no information on the Kodaira dimension of $\calC_d$ is known. We discuss some questions pertaining to the arithmetic of K3 surfaces raised by our results.
	\end{abstract}

	\maketitle
	

\section{Introduction}
	
	Let $\calC$ denote the $20$-dimensional coarse moduli space of smooth complex cubic fourfolds.  A very general $X \in \calC$ contains few algebraic surfaces. More precisely, the intersection
	\[
		A(X) := \HH^{2,2}(X) \cap \HH^4(X,\Z)
	\]
records the algebraic surfaces contained in $X$, because the integral Hodge conjecture holds for cubic fourfolds~\cite[Theorem~18]{VoisinHodge}, and $A(X) = \langle h^2\rangle$ for a very general $X\in \calC$, where $h$ denotes the restriction to $X$ of the hyperplane class in $\PP^5$.

In~\cite{HassettComp}, Hassett initiated the systematic study of the \defi{Noether-Lefschetz locus}
	\[
		\{X \in \calC : \rk(A(X)) > 1\}
	\]
consisting of smooth cubic fourfolds containing a surface not homologous to a complete intersection.  He defined a (labelled) \defi{special cubic fourfold} $X$ as a smooth cubic fourfold together with a rank two saturated lattice of algebraic cycles
	\[
		K := \langle h^2, T\rangle\subseteq A(X).
	\]
The \defi{discriminant} of $X$ is $d := |\disc(K)|$. The locus in $\calC$ of special cubic fourfolds with a labelling of discriminant $d$ is denoted $\calC_d$.

	We are starting to understand the geometry of the moduli spaces $\calC_d$. We already know from Hassett's work that $\calC_d$ is an irreducible divisor of $\calC$ if and only if $d > 6$ and $d \equiv 0$ or $2 \bmod 6$; it is empty otherwise~\cite[Theorem~1.0.1]{HassettComp}. Notably, Li and Zhang have recently determined the degrees of lifts of $\calC_d$ to $\PP^{55}$ (the projective space of all cubic fourfolds), for \emph{all} $d$, as coefficients of an explicitly computable modular form of weight $11$ and level $3$~\cite{LiZhang}. For low discriminants, Nuer~\cite{Nuer} has provided descriptions for the generic member of nonempty $\calC_d$ in the range $12 \leq d\leq 38$ and for $d = 44$, recovering known descriptions in the cases $d = 12$, $14$, and $20$ (the general member of $\calC_8$ is a cubic fourfold containing a plane). Nuer's work, combined with earlier results~\cites{Fano,Tregub,BD,HassettThesis}, also shows that $\calC_d$ is unirational (and thus has negative Kodaira dimension) for $8\leq d \leq 38$, and that $\calC_{44}$ has negative Kodaira dimension.  More recently, Lai~\cite{Lai} has shown that $\calC_{42}$ is uniruled, and thus has negative Kodaira dimension.
	
	Through well-documented connections with K3 surfaces, e.g.~\cites{HassettComp,AddingtonThomas,Huybrechts}, we know that some $\calC_d$ are of general type. For example, if $d \equiv 2 \bmod 6$, $4\nmid d$ and $9\nmid d$, and if every odd prime $p$ dividing $d$ satisfies $p \not\equiv 2 \bmod 3$, then $\calC_d$ is birational to the moduli space $\calK_d$ of polarized K3 surfaces of degree $d$~\cite[\S5]{HassettComp}. The spaces $\calK_d$ in turn are known to be of general type for $d > 122$ by groundbreaking work of Gritsenko, Hulek and Sankaran~\cite{GHSInventiones}. 
	
	In this paper, we study the birational type of the spaces $\calC_d$ for \emph{all large} $d$ simultaneously. Our main result is the following theorem.
		
	\begin{theorem} 
		\label{thm:MainThm}
		Let $\calC$ be the coarse moduli space of smooth cubic fourfolds, and let $\calC_d\subseteq \calC$ be the moduli space of special cubic fourfolds possessing a labelling of discriminant $d$.
		\smallskip
		\begin{enumerate}
			\item Assume that $d = 6n + 2$ for some integer $n$.
			\smallskip
			\begin{enumerate}
				\item If $n > 18$ and $n \notin \{20,21,25\}$ then $\calC_d$ is of general type;
				\smallskip
				\item If $n > 13$ and $n \neq 15$ then $\calC_d$ has nonnegative Kodaira dimension.
			\end{enumerate} 
			\smallskip
			\item Assume that $d = 6n$ for some integer $n$.
			\smallskip
			\begin{enumerate}
				\item If $n > 18$ and $n \notin \{20,22,23,25,30,32\}$ then $\calC_d$ is of general type;
				\smallskip
				\item If $n > 16$ and $n \notin \{18, 20, 22, 30\}$ then $\calC_d$ has nonnegative Kodaira dimension.
			\end{enumerate} 
		\end{enumerate}
	\end{theorem}	

	\begin{remark}
		In the notation of Theorem~\ref{thm:MainThm}, the results mentioned above say $\calC_d$ has negative Kodaira dimension for $d = 6n + 2$ and $1\leq n \leq 7$, as well as for $d = 6n$ and $2\leq n \leq 7$.  All together, these results and Theorem~\ref{thm:MainThm} leave only $20$ values of $d$ for which we cannot yet say anything about the Kodaira dimension of $\calC_d$. 
	\end{remark}


\subsection{Connection to the arithmetic of K3 surfaces}
	
	In joint work with McKinnie and Sawon~\cite[\S2.7]{MSTVA}, we showed that if $p \equiv 2 \bmod 3$ is prime, then there is a dominant morphism of quasi-projective varieties $\calC_{2p^2} \to \calK_2$. Furthermore, a very general point of $\calC_{2p^2}$ parametrizes a K3 surface $S$ of degree $2$ and geometric Picard rank $1$, \emph{together} with a nontrivial subgroup of order $p$ in the Brauer group $\Br S$. Suspecting that if $S$ is defined over $\Q$ then there should be only a few primes $p$ for which nontrivial classes in $(\Br S_{\C})[p]$ descend to $\Q$, we asked ourselves: is $\calC_{2p^2}(\Q) = \emptyset$ for $p \gg 0$?

	This question seems to be currently out of reach. The Bombieri-Lang conjectures, however, predict that if $\calC_{2p^2}$ is of general type, then rational points on $\calC_{2p^2}$ should be few and far between, at best. Our results imply that $\calC_{2p^2}$ is of general type for $p \geq 11$. We thus offer the following question as a challenge.

	\begin{question}
		Does there exist a K3 surface $S/\Q$ of degree $2$ with geometric Picard rank $1$ such that $(\Br S)[p] \neq 0$ for some prime $p \geq 11$?
	\end{question}
	
	Not all subgroups of order $p$ in the transcendental Brauer group of a K3 surface of degree $2$ are parametrized by points on some $\calC_{2p^2}$. Some, for example, are parametrized by points on $\calK_{2p^2}$ through a dominant morphism $\calK_{2p^2} \to \calK_2$ by work of Mukai (see~\cite{Mukai},\cite[\S2.6]{MSTVA}). However, the results of~\cite{GHSInventiones} show that $\calK_{2p^2}$ is also of general type for $p \geq 11$.
	

\subsection{Relation to orthogonal modular varieties}
	\label{ss:modvars}

	The proof of Theorem~\ref{thm:MainThm} builds on several techniques developed by Gritsenko, Hulek, and Sankaran over the last ten years~\cites{GHSInventiones,GHSCompositio,GHSDifferential,GHSHandbook}. We explain the basic set-up.
	
	Let $L$ be an even lattice of signature $(2,m)$ with $m \geq 9$. Write
	\[
		\calD_{L} = \{ [x] \in \PP(L\otimes \C) : (x,x) = 0, (x,\bar x) > 0\}^+
	\]
for one of the two components of its associated period domain, and let $\calD^\bullet_L$ denote the affine cone of $\calD_L$. Write $O^+(L)$ for the subgroup of the orthogonal group of $L$ that preserves $\calD_{L}$, and $\widetilde{O}^+(L)$ for the intersection of $O^+(L)$ with the stable orthogonal group $\widetilde{O}(L)$ of $L$ (see~\S\S\ref{ss:Lattices} and~\ref{ss:PeriodDomains} for details).  For example, the nonspecial cohomology $K_d^\perp(-1) \subseteq \HH^4(X,\Z)(-1)$ of a cubic fourfold $X$ is an even lattice of signature $(2,19)$.
	
	Our point of departure is the observation that $\calC_d$ is birational to the orthogonal modular variety $\Gamma \backslash \calD_{K_d^\perp(-1)}$ for an appropriate choice of finite index subgroup $\Gamma$ of $O^+(K_d^\perp(-1))$; see~\S\ref{ss:ModuliSpecialFourfolds} for details on the nature of $\Gamma$. This observation allows us to apply the ``low-weight cusp form trick'' of Gritsenko, Hulek and Sankaran (\S\ref{ss:LowWeightCusp}), and thus reduce Theorem~\ref{thm:MainThm} to finding a single cusp form of weight $< 19$ and character $\det$, with respect to the group $\Gamma$, that vanishes along the ramification divisor of the modular projection $\pi\colon \calD_{K_d^\perp(-1)} \to \Gamma \backslash \calD_{K_d^\perp(-1)}$.


\subsection{From cusp forms to lattice embeddings}
	\label{ss:cusplat}

	There is a strategy, originating in work of Kond{\=o}~\cite{KondoII} and brought to maturity in~\cite{GHSHandbook}, to find the needed low-weight cusp form. Let $L_{2,26} = U^{\oplus 2}\oplus E_8(-1)^{\oplus 3}$ be the \defi{Borcherds lattice}; here $U$ denotes a hyperbolic plane and $E_8(-1)$ is the negative definite version of the $E_8$ lattice. In~\cite{Borcherds}, Borcherds introduced a nonzero modular form $\Phi_{12}\colon \calD^\bullet_{L_{2,26}} \to \C$ of weight $12$ and character $\det$ with respect to the stable orthogonal group $\widetilde{O}^+(L_{2,26})$ of $L_{2,26}$. A primitive embedding of lattices $K_d^\perp(-1)\hookrightarrow L_{2,26}$ induces an inclusion of affine cones $\calD^\bullet_{K_d^\perp(-1)} \hookrightarrow \calD^\bullet_{L_{2,26}}$, so one might hope to restrict $\Phi_{12}$ to $\calD^\bullet_{K_d^\perp(-1)}$ to get the desired cusp form. However, the form $\Phi_{12}$ vanishes, with multiplicity $1$, along $\calD_{\langle r\rangle^\perp}^\bullet$ for any $r \in L_{2,26}$ such that $r^2 = -2$. If $(r,K_d^\perp(-1)) = 0$, then $\Phi_{12}|_{\calD^\bullet_{K_d^\perp(-1)}}$ will also vanish identically. Thus, we are forced to divide $\Phi_{12}$ by linear functions parametrized by \emph{pairs} of roots in $L_{2,26}$ orthogonal to $K_d^\perp(-1)$. The resulting form
	\[
		\Phi|_{K_d^\perp(-1)} := \frac{\Phi_{12}(Z)}{\displaystyle \prod_{r \in R_{-2}(K_d^\perp(-1))/\{\pm 1\}}(Z,r)}\Bigg|_{\calD^\bullet_{K_d^\perp(-1)}}
	\]
is a cusp form for $\widetilde{O}^+(K_d^\perp(-1))$ provided $N := \#R_{-2}(K_d^\perp(-1))/\{\pm 1\} > 0$, often referred to as a \defi{quasi-pullback} of $\Phi_{12}$ (see~\cite[\S8]{GHSHandbook}). The form $\Phi|_{K_d^\perp(-1)}$ has weight $12 + N$; the low-weight cusp form trick requires that $N < 7$, leaving us with the problem that occupies the bulk of the paper:

	\begin{problem}
		\label{prod:embedding}
		Construct a primitive embedding $K_d^\perp(-1) \hookrightarrow L_{2,26}$ in such a way that
		\[
			0 < \#\{r \in L_{2,26} : r^2 = -2, (r,K_d^\perp(-1)) = 0\} < 14.
		\]
	\end{problem}	
	

\subsection{Contributions}
	\label{ss:contributions}

Solving problems akin to Problem~\ref{prod:embedding} requires several major developments in~\cites{GHSInventiones,GHSCompositio,GHSDifferential,GHSHandbook}, but the class of orthogonal varieties we consider presents additional obstacles not encountered in previous problems of this type.

	There is a lattice isomorphism $K_d^\perp(-1) \isom B_n \oplus U \oplus E_8(-1)^{\oplus 2}$, where $B_n$ is the rank three lattice of signature $(1,2)$ with intersection pairing given by
	\[
		\begin{pmatrix}
			-2 &  1 &  0 \\
			1  & -2 &  \eps \\
			0  &  \eps & 2n
		\end{pmatrix}
		\quad\textrm{and}\quad
		\eps = 
		\begin{cases}
			0 & \textrm{if } d = 6n, \\
			1 & \textrm{if } d = 6n + 2,
		\end{cases}
	\]
(see~\S\ref{s:TheLattice}). Results of Nikulin reduce Problem~\ref{prod:embedding} to embedding $B_n \hookrightarrow U\oplus E_8(-1)$ with image orthogonal to only a few vectors of length $-2$ in $U\oplus E_8(-1)$. This task is more strenuous than the analogous embedding problem in~\cite[\S7]{GHSInventiones}, for several reasons.
	\smallskip
	\begin{enumerate}
		\item The lattice $U \oplus E_8(-1)$ is indefinite, so there is no \emph{a priori} bound, independent of $n$, for the number of vectors of length $-2$ in $U\oplus E_8(-1)$ orthogonal to the image of $B_n$ (signature considerations do imply that the number of such vectors is finite). The calculations in \S\ref{S:latticeEng} give rise to a class of embeddings for which we can certify an upper bound that is independent of $n$, at least for sufficiently large $n$.
		\smallskip
	\end{enumerate}
	When $d = 6n + 2$, we reduce the embedding problem to the positivity of the coefficients of a difference of theta series, as in~\cite[\S7]{GHSInventiones}. Two difficulties arise in our case:
	\smallskip
	\begin{enumerate}\setcounter{enumi}{1}
		\item We must consider some lattices of odd rank, making the problem of estimating the size of the coefficients of their theta functions a difficult one. Fortunately, the odd-rank lattices we consider are unique in their genus, so later work of Gritsenko, Hulek and Sankaran~\cite[\S5]{GHSCompositio}, and Yang~\cite{Yang} can be applied to control these coefficients.
		\smallskip
		\item We lack explicit \emph{closed-form} bases for the spaces of modular forms containing the theta series we consider. This forces a delicate asymptotic analysis (\S\S\ref{ss:M1Bound}--\ref{ss:ProofThm1.1}) to prove that $\calC_d$ is of general type for $n \geq 1207$, leaving a tractable linear algebra problem that a computer can handle for the remaining smaller values of $n$.  
		\smallskip
	\end{enumerate}

	A new feature in the case $d\equiv 2 \bmod 6$ is the prominent use of dual $\Q$-lattices to home in on good embeddings $K_d^\perp(-1) \hookrightarrow L_{2,26}$. See, for example, the statement of Problem~\ref{prob:lattices} and the types of lattices involved in Theorem~\ref{thm:reduction}. Dual $\Q$-lattices give us theoretical and computational leverage in proving Theorem~\ref{thm:MainThm}. Theoretically, we use them to refine estimates for the size of coefficients of the theta functions that we need (\S\S\ref{ss:M1Bound}--\ref{ss:ProofThm1.1}). Computationally, they allow us to efficiently search for embeddings when $n < 1207$; the required calculations are feasible on a single desktop machine over the course of a few days\footnote{In addition, \emph{verifying} the result of the calculations takes only a few seconds on a desktop machine.}. \\

The case $d \equiv 0 \bmod 6$ has a purely lattice-theoretic proof, i.e., no theta functions are involved, but it presents a different set of challenges, the most important of which is akin to that in~\cite{GHSDifferential}: the modular group $\Gamma$ in this case contains $\widetilde{O}^+(K_d^\perp(-1))$ as a \emph{subgroup of index $2$}, so the cusp form $\Phi|_{K_d^\perp(-1)}$ is not \emph{a priori} modular with respect to the correct group. Our solution to this problem for large $d$ requires exhibiting a vector $v \in E_6$ with certain properties (see Lemmas~\ref{lem:modularforbiggergroup} and~\ref{lem:pexists}), whose length can be expressed as the sum of three \emph{nonzero} squares. To this end, we use a strengthening of Legendre's three-squares theorem. This strengthening, however,  requires the Generalized Riemann Hypothesis (GRH). We circumvent GRH by keeping small the number of $\calC_d$ whose Kodaira dimension we must calculate using a computer; see the proof of Theorem~\ref{thm:Dis0mod6}.
	

\subsection{Related work}

	Besides results of Hassett, Lai, and Nuer already mentioned, and results implied by work of Gritsenko, Hulek, and Sankaran, it is worth mentioning a recent remarkable result of Ma~\cite{Ma}. Using Gritsenko's method of Jacobi lifting~\cite{Gritsenko}, Ma shows that all but finitely many stable orthogonal modular varieties are of general type.  In particular, his result implies that $\calC_d$ is of general type for $d \gg 0$ when $d \equiv 2 \bmod 6$. Ma asserts that ``with a huge amount of computation'' his result could be made effective.  We have not attempted to do this because the case $d \equiv 0 \bmod 6$ is not covered by his proof.

	One might have hoped that Theorem~\ref{thm:MainThm} could shed light on the Kodaira dimension of the moduli of polarized K3 surfaces $\calK_d$ for some values of $d$ not covered by~\cite{GHSInventiones}, because our proof does not rely on the specific embeddings used in~\cite{GHSInventiones}. Alas, our results are compatible on the nose with those of~\cite{GHSInventiones}.


\subsection*{Outline}

	In \S\ref{s:Preliminaries}, we set up notation and review tools necessary for the proof of Theorem~\ref{thm:MainThm}. After fixing some terminology in \S\ref{ss:Lattices}, we use \S\ref{ss:reps} to review some results in~\cite{GHSCompositio} and~\cite{Yang} on the representation of integers by quadratic forms needed for the case $d \equiv 2 \bmod 6$. In \S\S\ref{ss:PeriodDomains}--\ref{ss:ModuliSpecialFourfolds} we discuss moduli of cubic fourfolds and the birational map $\calC_d \dasharrow \Gamma \backslash \calD_{K_d^\perp(-1)}$ mentioned in~\S\ref{ss:modvars}. In~\S\ref{ss:LowWeightCusp} we explain the low-weight cusp form trick that reduces Theorem~\ref{thm:MainThm} to Problem~\ref{prod:embedding}.
	
	In \S\ref{s:TheLattice} we show that $K_d^\perp(-1)$ is isomorphic to the lattice $B_n\oplus U \oplus E_8(-1)^{\oplus 2}$, and in \S\ref{S:latticeEng} we explain some natural constraints on the embedding $B_n \hookrightarrow U \oplus E_8(-1)$ that allow us to control the type and total number of $(-2)$-vectors orthogonal to the image of $B_n$.
	
	We treat the case $d = 6n + 2$ in~\S\ref{s:2mod6}. After introducing certain dual $\Q$-lattices that make tractable the problem of embedding $K_d^{\perp}(-1)$ into $L_{2,26}$, we prove Theorem~\ref{thm:reduction}, a result in the spirit of~\cite[Theorem~7.1]{GHSInventiones}, which reduces the problem for large $d$ to an inequality of coefficients of certain theta functions. The asymptotic analysis in \S\S\ref{ss:M1Bound}--\ref{ss:ProofThm1.1} is then used to show this inequality is satisfied for $n \geq 1207$ in Theorem~\ref{thm:Dis2mod6}. In the proof of Theorem~\ref{thm:Dis2mod6} we also outline the computation that exhibits the necessary embedding for the values $n < 1207$ appearing in Theorem~\ref{thm:MainThm}.
	
	Finally, we treat the case $d = 6n$ in \S\ref{s:0mod6}. In~\S\ref{ss:embeddingLwithprops} we give a criterion ensuring the modular form $\Phi|_{K_d^\perp(-1)}$ is modular with respect the correct monodromy group (Lemma~\ref{lem:modularforbiggergroup}). We also give a flexible criterion meant to ensure that the embedding $K_d^{\perp}(-1) \hookrightarrow L_{2,26}$ is primitive when $d = 6n$ (Lemma~\ref{lem:primitivity}); this criterion is essential to handle small values of $n$ in Theorem~\ref{thm:MainThm}, particularly $n < 44$. Lemma~\ref{lemma:lessthan14} shows that the criterion ensuring $\Phi|_{K_d^\perp(-1)}$ is modular with respect the correct group also ensures that the image of $K_d^\perp(-1)$ in $L_{2,26}$ is orthogonal to only a few $(-2)$-vectors. This allows us to complete the embedding analysis in Theorem~\ref{thm:Dis0mod6}. In \S\ref{ss:ramification} we check that the form $\Phi|_{K_d^\perp(-1)}$ vanishes along the ramification divisor of the orthogonal modular projection, completing the proof of Theorem~\ref{thm:MainThm}.


\subsection*{Acknowledgements}

	We are grateful to Brendan Hassett for many conversations where he patiently answered our questions, and for his constant encouragement. We thank Klaus Hulek for an illuminating discussion at the Simons Symposium ``Geometry over non-closed Fields'' in March of 2015. We also thank an anonymous referee for their careful reading of the manuscript, and for pertinent suggestions that improved the exposition of the paper. Tanimoto was partially supported by Lars Hesselholt's Niels Bohr Professorship. V\'arilly-Alvarado was supported by NSF grant DMS-1103659 and NSF CAREER grant DMS-1352291. Computer calculations were carried out in {\tt Magma}~\cite{Magma}.
	

\section{Preliminaries}
\label{s:Preliminaries}


\subsection{Lattices}
\label{ss:Lattices}

	By a \defi{lattice} $L$ we mean a free abelian group of finite rank endowed with a symmetric, bilinear, nondegenerate integral pairing $(\cdot\,,\cdot)\colon L \times L \to \Z$; if we allow the pairing to take values in $\Q$ instead of $\Z$, then we refer to $L$ as a \defi{$\Q$-lattice}. We write $O(L)$ for the group of orthogonal transformations of a lattice $L$.  For $m \in \Z$, we denote by $L(m)$ the lattice whose underlying abelian group coincides with that of $L$, and whose pairing is that of $L$ multiplied by $m$. For an inclusion of ($\Q$-)lattices $L\subseteq M$, we write $L^\perp_M$ for the orthogonal complement of $L$ in $M$. The \defi{length} of a vector $x$ in a ($\Q$-)lattice $L$ is the quantity $(x,x)$.

	Let $L^\vee = \Hom(L,\Z)$ denote the dual abelian group of a lattice $L$; it is a $\Q$-lattice.  We use the inclusion $L \to L^\vee$ given by $\lambda \mapsto [\mu \mapsto (\lambda,\mu)]$ to define the \defi{discriminant group} $D(L) = L^\vee/L$ of $L$; it is a finite abelian group. If $L$ is \defi{even}, i.e., if $(x,x) \in 2\Z$ for all $x \in L$, then define the \defi{discriminant form} of $L$ by
	\begin{equation*}
		q_L\colon D(L) \to \Q/2\Z, \qquad
		x + L \mapsto (x,x) \bmod 2\Z.
	\end{equation*}
Write $O(D(L))$ for the orthogonal group of this quadratic form, and $\widetilde O(L)$ for the \defi{stable orthogonal} group of $L$, i.e., the kernel of the natural map
	\[
		O(L) \to O(D(L)).
	\]

	The \defi{theta series} associated to a definite $\Q$-lattice is the generating function
	\[
		\Theta_L(q) = \sum_{x \in L}q^{( x,x)} = \sum_m N_L(m) q^{m}
	\]
that counts the number $N_L(m)$ of vectors $x \in L$ of given length $( x,x) = m$. If $B$ is a Gram matrix for a chosen basis of $L$, then we sometimes write $N_B(m)$ instead of $N_L(m)$ (these numbers do not depend on the chosen basis).


\subsection{Representation of integers by lattices}
\label{ss:reps}

	Let $L$ be an even positive definite lattice of rank $r$. Let $B$ denote the Gram matrix corresponding to a chosen basis for $L$, and write $|B|$ for the determinant of $B$. Let $S(x) = \frac{1}{2} x^t B x$ be the corresponding quadratic form, with $x \in \Z^r$. If $S$ is unique in its genus and $r \geq 3$, then the number $r(t, S) = N_B(t)$ of representations of $t\in \Z$ by $S$ is a product of local densities
	\begin{equation}
		\label{eq:prodlocaldensity}
		r(t, S) = \prod_{p \leq \infty} \alpha_p(t,S)
	\end{equation}
	where
	\[
		\alpha_p(t,S) = \lim_{a\to\infty} p^{-a(r-1)}\#\{ x \in (\Z/p^a\Z)^r : S(x) \equiv t \bmod p^a \},
	\]
	for $p$ a finite prime, and
	\[
		\alpha_\infty(t,S) = (2\pi)^{r/2}\Gamma\left(\frac{r}{2}\right)^{-1}t^{r/2 - 1}|B|^{-1/2}.
	\]

	The following theorem helps compute the right hand side of~\eqref{eq:prodlocaldensity}, depending on the parity of the rank~$r$ of $B$. We need  the \defi{Zagier $L$-function}, which for $\Delta \equiv 0, 1 \bmod 4$ is defined by
	\[
		L(s, \Delta) = \frac{\zeta(2s)}{\zeta(s)} \sum_{n=1}^{\infty} \frac{b_n(\Delta)}{n^s},
	\]
where $b_n(\Delta) = \#\{ x \bmod 2n: x^2 \equiv \Delta \bmod 4n \}$.

	\begin{theorem}
		\label{thm:representationsbyintegers}
		For $x \in \Z^r$, let $S(x) = \frac{1}{2} x^t B x$ be the quadratic form corresponding to a symmetric even integral positive definite $r\times r$ matrix $B$. Assume that $S$ is unique in its genus. Decompose a fixed integer $t$ as a product $t = t_Bt_1t_2^2$, where $t_1$ is square free, $(t_1t_2^2,|B|) = 1$, and $t_B$ divides some power of $|B|$. Put
		\[
			D = 
			\begin{cases}
				\disc\left( \Q\left(\sqrt{(-1)^{\frac{r-1}{2}}2t|B|}\right) \right) & \text{if $r$ is odd;}\\
				(-1)^{r/2}|B| & \text{if $r$ is even.}
			\end{cases}
		\]
Let $\chi_D(p) = \big(\frac{D}{p}\big)$ and $\chi_{4D}(a) = \big(\frac{4D}{a}\big)$ denote the usual Jacobi symbols. Then
		\[
			r(t,S) = 
			\begin{cases}
				\begin{aligned}
					(2\pi)^{r/2}&\Gamma\left(\frac{m}{2}\right)^{-1}t^{r/2 - 1}|B|^{-1/2}L\left(\frac{r-1}{2},Dt_2^2\right)\zeta(r-1)^{-1} \\
					& \quad\times \prod_{p\, \mid\, |B|}\frac{1 - \chi_D(p)p^{(1 - r)/2}}{1 - p^{1-r}}\alpha_p(t,S),
				\end{aligned}
				&  \text{if $r$ is odd;}\\
				\displaystyle\bigg(\sum_{a \mid t}\chi_{4D}(a)a^{1 - (r/2)}\bigg)L\left(\frac{r}{2},\chi_{4D}\right)^{-1}\alpha_\infty(t,S)\prod_{p \mid 2D}\alpha_p(t,S)
				& \text{if $r$ is even.}
			\end{cases}
		\]
Here $L\left(\frac{r-1}{2},Dt_2^2\right)$ denotes the Zagier $L$-function, and $L\left(\frac{r}{2},\chi_{4D}\right)$ is a Dirichlet $L$-function.
	\end{theorem}

	\begin{proof}
		See~\cite[Theorem~5.1]{GHSCompositio} for $r$ odd, and~\cite[(11.74)]{Iwaniec} for $r$ even.
	\end{proof}

	To apply Theorem~\ref{thm:representationsbyintegers}, we shall need estimates for the local densities $\alpha_p(t,S)$. We explain how to obtain such estimates, following Yang~\cite{Yang}. We assume throughout that $p^{-1}S$ is not half-integral. For $p \neq 2$, the quadratic form $S$ is $\Z_p$-equivalent to a diagonal form with Gram matrix
	\begin{equation}
		\label{eq:ZpNormalization}
		\diag(\epsilon_1 p^{l_1},\dots,\epsilon_m p^{l_m}) \quad\textrm{with } \epsilon_i \in \Z_p^\times, \textrm{ and } l_1\leq \cdots \leq l_m.
	\end{equation}
Since $p^{-1}S$ is not half-integral, we have $l_1 = 0$. For $k \in \Z_{\geq 0}$, set
	\[
		L(k,1) = \{1 \leq i \leq m : l_i - k < 0 \textrm{ is odd}\}\quad \textrm{and}\quad l(k,1) = \#L(k,1).
	\]
For $t = up^a$ with $u \in \Z_p^\times$ and $a \in \Z$, set
	\[
		f(t) = 
		\begin{cases}
			\displaystyle
			-\frac{{1}}{p} & \textrm{if $l(a+1,1)$ is even}, \\
			\displaystyle
			\left(\frac{u}{p}\right)\frac{1}{\sqrt{p}} & \textrm{otherwise}.
		\end{cases}
	\]
Finally, define
	\[
		d(k) = k + \frac{1}{2}\sum_{l_i < k}(l_i - k) 
		\quad\text{and}\quad
		v(k) = \left(\frac{-1}{p}\right)^{[l(k,1)/2]}\prod_{i \in L(k,1)}\left(\frac{\epsilon_i}{p}\right).
	\]
Then, by~\cite[p.\ 317, Proof of Theorem~3.1]{Yang} we have
	\[
		\alpha_p(t,S) = 1 + R_1(t,S),
	\]
where
	\begin{equation}
		\label{eq:R1pOdd}
			R_1(t,S) = (1 - p^{-1})\bigg(\sum_{\substack{0 < k \leq a \\ l(k,1) \textrm{ even}}}
			v(k) p^{d(k)}\bigg) + v(a+1)p^{d(a+1)}f(t).
	\end{equation}

	The case $p=2$ is more involved. Suppose that the quadratic form $S$ is $\Z_2$-equivalent to a quadratic form with Gram matrix
	\begin{equation}
		\label{eq:diagZ2}
		\diag(\epsilon_12^{l_1}, \cdots, \epsilon_L2^{l_L}) \oplus \left( \bigoplus_{i=1}^M \epsilon_i'2^{m_i } \begin{pmatrix}0 & \frac{1}{2} \\ \frac{1}{2} & 0\end{pmatrix} \oplus \bigoplus_{j=1}^N \epsilon_j''2^{n_j}\begin{pmatrix}1 & \frac{1}{2} \\ \frac{1}{2} & 1\end{pmatrix} \right).
	\end{equation}
For each integer $k>0$, we use the following notation:
	\begin{align*}
			L(k,1) &= \{1 \leq i \leq L : l_i - k < 0 \textrm{ is odd}\}\quad \textrm{and}\quad l(k,1) = \#L(k,1), \\
			p(k) &= (-1)^{\sum_{n_j < k} (n_j-k)},\\
			\epsilon(k) &= \prod_{h \in L(k-1, 1)}  \epsilon_h,\\
			d(k) & = k+\frac{1}{2} \sum_{l_h < k-1} (l_h-k+1) + \sum_{m_j < k} (m_i-k) + \sum_{n_j <k} (n_j-k),\\
			\delta(k) &=
			\begin{cases} 
				0 & \text{if $l_h = k-1$ for some $h$,}\\
				1 & \text{otherwise;}
			\end{cases} \\
			\intertext{}
			\left(\frac{2}{x}\right) &=
			\begin{cases} 
				(2,x)_2 & \text{if $x \in \Z_2^\times$,}\\
				0 & \text{otherwise,}
			\end{cases} \\
			\Char_X(f) &= \textrm{the characteristic function of $f$ on a set $X$}
	\end{align*}
Then, by~\cite[Theorem~4.1]{Yang} we have
	\[
		\alpha_2(t,S) = 1 + R_1(t,S),
	\]
where, for $t = u2^a$ with $u \in \Z_2^\times$ and $a\in \Z$,
	\begin{equation}
		\label{eq:R1Z2}
		\begin{split}
			R_1(t,S) &= \sum_{\substack{l < k \leq a +3 \\ l(k-1,1) \textrm{ odd}}}
			\delta(k) p(k) \left(\frac{2}{\mu\epsilon(k)}\right)2^{d(k) - 3/2} \\
			&\qquad + \sum_{\substack{l < k \leq a +3 \\ l(k-1,1) \textrm{ even}}}
			\delta(k) p(k) \left(\frac{2}{\mu\epsilon(k)}\right)2^{d(k) - 1}e^{-2\pi i (\mu/8)}
			\Char_{4\Z_2}(\mu),
		\end{split}
	\end{equation}
and $\mu = \mu_k(t)$ is given by
	\[
		\mu_k(t) = u2^{a - k + 3} - \sum_{l_h < k-1}\epsilon_h.
	\]


\subsection{Period domains}
\label{ss:PeriodDomains}

	Let $L$ be a lattice of signature $(2,m)$, with $m > 1$. Extending the pairing on $L$ by $\C$-linearity, we let
	\[
		\Omega_L = \{ [x] \in \PP(L\otimes \C) : (x,x) = 0, (x,\bar x) > 0\}
	\]
be the \defi{period domain} associated to $L$. The space $\Omega_L$ consists of two connected components $\calD_L$ and $\calD_L'$ interchanged by complex conjugation. We write $O^+(L)$ for the subgroup of elements in $O(L)$ that preserve $\calD_L$, and set
	\[
		\widetilde O^+(L) = \widetilde O(L) \cap O^+(L).
	\]

	Let $\Gamma \subset O^+(L)$ be a subgroup of finite index. We define the \defi{modular variety} of $L$ with respect to $\Gamma$ as
	\[
		\mathcal F_L(\Gamma) = \Gamma \backslash \calD_L.
	\]


\subsection{Moduli of cubic fourfolds}

	In the next two subsections we introduce the modular varieties at the heart of this paper; a good reference for this material is Hassett's paper~\cite{HassettComp}.

	Two cubic fourfolds are isomorphic if and only if they are congruent under the natural action of $\SL_6$. Let $V \subseteq \PP^{55}$ denote the locus of smooth cubic fourfolds within the  projective space of all cubic fourfolds.  The GIT quotient $\calC := V//\SL_6$ is a coarse moduli space for smooth cubic fourfolds; it is a quasi-projective variety of dimension $20$.

	Let $X$ be a smooth complex cubic fourfold. Cup product makes the singular cohomology group $\HH^4(X,\Z)$ into a lattice, which is abstractly isomorphic to
	\[
		\Lambda := (+1)^{\oplus 21} \oplus (-1)^{\oplus 2};
	\]
see~\cite[Proposition~2.1.2]{HassettComp}. Let $h^2 \in \Lambda$ be a primitive vector with $(h^2,h^2) = 3$, and let 
	\[
		\Gamma = \{g \in O(\Lambda(-1)) : g(h^2) = h^2\}.
	\]
By~\cite[Proposition~2.1.2]{HassettComp}, the lattice 
	\[
		\Lambda_0(-1) := \langle h^2\rangle^\perp(-1) \subseteq \Lambda(-1)
	\]
is isomorphic to
	\[
		B(-1)\oplus U^{\oplus 2} \oplus E_8(-1)^{\oplus 2},
	\]
where $U$ is the hyperbolic plane, $E_8(-1)$ is the unique negative definite rank eight unimodular lattice, and $B$ is the rank two lattice with pairing
	\begin{equation}
		\label{eq:B}
		\begin{pmatrix}
			2 & 1 \\
			1 & 2
		\end{pmatrix}.
	\end{equation}
The lattice $\Lambda_0(-1)$ has signature $(2,20)$.  The group $\Gamma$ acts on $\Omega_{\Lambda_0(-1)}$; let $\Gamma^+ \subseteq \Gamma$ be the index two subgroup of $\Gamma$ that preserves the component $\calD_{\Lambda_0(-1)}$ of $\Omega_{\Lambda_0(-1)}$. The orbit space $\calD := \Gamma^+\backslash\calD_{\Lambda_0(-1)}$ is an analytic space, and also a quasi-projective variety of dimension $20$ by results of Baily and Borel~\cite{BB}.  It is called the \defi{global period domain} for smooth cubic fourfolds. There is a \defi{period map}
	\[
		\tau\colon \calC \to \calD,
	\]
which is an open immersion of complex analytic spaces~\cite{Voisin}, whose image was determined in~\cites{Looijenga,Laza}. The map $\tau$ is algebraic~\cite[Proposition~2.2.2]{HassettComp}, so we may view $\calC$ as a Zariski open subset of $\calD$.


\subsection{Moduli of special cubic fourfolds}
	\label{ss:ModuliSpecialFourfolds}

	We define a rank two lattice $K_d \subseteq \Lambda$ whose discriminant has absolute value $d$, and which contains $h^2$, as follows.  We let $K_d = \langle h^2, T\rangle$, with intersection matrix
	\[
		\begin{pmatrix}
			3 & 0 \\
			0 & 2n
		\end{pmatrix}\quad\textrm{or}\quad
		\begin{pmatrix}
			3 & 1 \\
			1 & 2n+1
		\end{pmatrix}
	\]
according to whether $d = 6n$ for some $n \in \Z_{\geq 2}$, or $d = 6n + 2$ for some $n \in \Z_{\geq 1}$, respectively. The inclusion $K_d^\perp(-1) \subseteq \Lambda_0(-1)$ induces an inclusion $\calD_{K_d^\perp(-1)}\subseteq \calD_{\Lambda_0(-1)}$. Let
	\[
		\Gamma_d^+ = \{g \in \Gamma^+ : g(K_d) = K_d\},
	\]
and let $\overline{\Gamma}_d^+$ be the image of the restriction map
	\[
		\Gamma_d^+ \to O^+(K_d^\perp(-1)).
	\]
The group $\overline{\Gamma}_d^+$ has finite index in $O^+(K_d^\perp(-1))$: by~\cite[Corollary~1.5.2]{Nikulin}, the stable orthogonal group $\widetilde{O}^+(K_d^\perp(-1))$ includes in $\overline{\Gamma}_d^+$, and $\widetilde{O}^+(K_d^\perp(-1))$ has finite index in $O^+(K_d^\perp(-1))$. In fact, the groups $\widetilde{O}^+(K_d^\perp(-1))$ and $\overline{\Gamma}_d^+$ coincide when $d \equiv 2 \bmod 6$, and $\widetilde{O}^+(K_d^\perp(-1))$ is an index $2$ subgroup in $\overline{\Gamma}_d^+$ when $d \equiv 0 \bmod 6$; see~\cite[Proposition~5.2.1]{HassettComp}.

	The modular variety
	\[
		\calD_d := \overline{\Gamma}_d^+\backslash \calD_{K_d^\perp(-1)}
	\]
is called the global period domain for labelled special cubic fourfolds of discriminant $d$. The natural map $\calD_d \to \calD$ is birational onto its image, which contains the space $\calC_d$ as a Zariski open subset; see~\cite[p.\ 7]{HassettComp}.


\subsection{The low-weight cusp form trick}
\label{ss:LowWeightCusp}

	Let $L$ be a lattice of signature $(2,m)$, with $m > 1$. Write $\calD_L^\bullet$ for the affine cone of $\calD_L$. For a subgroup $\Gamma \subseteq O^+(L)$ of finite index, and a character $\chi\colon\Gamma \to \C^\times$, a \defi{modular form of weight $k$ and character $\chi$} is a function $F\colon \calD_L^\bullet \to \C$ such that
	\smallskip
	\begin{enumerate}
		\item $F(tZ) = t^{-k}F(Z)$ for all $Z \in \calD_L^\bullet$ and all $t \in \C^\times$;
		\smallskip
		\item $F(\gamma\cdot Z) = \chi(\gamma)F(Z)$ for all $Z \in \calD_L^\bullet$ and all $\gamma \in \Gamma$;
	\end{enumerate}

	\begin{theorem}[{\cite[Theorem 1.1]{GHSInventiones}}]
		\label{thm:LowWeightCusp}
		Let $L$ be an integral lattice of signature $(2,m)$ with $m \geq 9$, and let $\Gamma \subseteq O^+(L)$ be a subgroup of finite index. The modular variety $\calF_L(\Gamma)$ is of general type if there exists a nonzero cusp form $F \in S_a(\Gamma,\chi)$ of weight $a < m$ and character $\chi \in \{1,\det\}$ that vanishes along the ramification divisor of the projection $\pi\colon \calD_L \to \calF_L(\Gamma)$.

		If $S_m(\Gamma,\det) \neq 0$ then the Kodaira dimension of $\calF_L(\Gamma)$ is nonnegative. \qed
	\end{theorem}

	One way to produce a cusp form $F \in S_a(\Gamma,\chi)$ is to leverage a modular form constructed by Borcherds in~\cite[\S10, Example~2]{Borcherds}. Let $L_{2,26}$ denote the \defi{Borcherds lattice} $U^{\oplus 2} \oplus E_8(-1)^{\oplus 3}$. The \defi{Borcherds form} $\Phi_{12} \in M_{12}(O^+(L_{2,26}),\det)$ is the unique (up to scaling) nonzero modular form of weight $12$ and character $\det$ with respect to the group $O^+(L_{2,26})$.  The form $\Phi_{12}$ vanishes, with multiplicity one, along certain divisors of $\calD_{L_{2,26}}^\bullet$, determined by the $(-2)$-vectors of the lattice $L_{2,26}$. More precisely, $\Phi_{12}(Z) = 0$ if and only if there exists an $r \in L_{2,26}$ with $r^2 = {-2}$ such that $(r,Z) = 0$.

	An embedding of lattices $L \hookrightarrow L_{2,26}$ gives an inclusion $\calD_L^\bullet \hookrightarrow \calD_{L_{2,26}}^\bullet$. If $r \in L_{2,26}$ is a $(-2)$-vector such that $(r,L) = 0$, then the restriction $\Phi_{12}|_{\calD_L^\bullet}$ vanishes identically.  Thus, to get a nonzero modular form, we first need to divide $\Phi_{12}$ by the linear functions $Z\mapsto (Z,r)$, as $r$ ranges over the $(-2)$-vectors of $L_{2,26}$, up to sign, that are orthogonal to $L$. Every time we divide out by a linear function, the weight of the resulting modular form increases by one.  We summarize this discussion in the following theorem (see~\cite[Theorem~8.2 and Corollary~8.12]{GHSHandbook}).

	\begin{theorem}
		\label{thm:quasipullback}
		Let $L$ be a lattice of signature $(2,m)$ with $m \geq 3$.  Let $L \hookrightarrow L_{2,26}$ be a primitive embedding, giving rise to an inclusion $\calD_L^\bullet \hookrightarrow \calD_{L_{2,26}}^\bullet$. The set
		\[
			R_{-2}(L) = \{r \in L_{2,26} : r^2 = -2, (r,L) = 0\},
		\]
is finite and has even cardinality; we let $N(L) = \#R_{-2}(L)/2$. The form
		\[
			\Phi|_L := \frac{\Phi_{12}(Z)}{\displaystyle \prod_{r \in R_{-2}(L)/\{\pm 1\}}(Z,r)}\Bigg|_{\calD^\bullet_L}
		\]
is modular of weight $12 + N(L)$ and character $\det$, with respect to the group $\widetilde{O}^+(L)$, i.e., $\Phi|_L \in M_{12+N(L)}(\widetilde{O}^+(L),\det)$. If $N(L) > 0$, then $\Phi|_L$ is a cusp form.
		\qed
	\end{theorem}


\section{The lattice $K_d^\perp(-1)$}
\label{s:TheLattice}


	\begin{proposition}
		\label{prop:Kdperp}
		Let $K_d \subseteq \Lambda$ be the lattice of~\S\ref{ss:ModuliSpecialFourfolds}. Then $K_d^{\perp}(-1)$ is isomorphic to the rank $21$ lattice
		\[
			B_n \oplus U \oplus E_8(-1)^{\oplus 2} 
		\]
of signature $(2,19)$, where $B_n$ is the rank $3$ lattice with intersection pairing
		\begin{equation}
			\label{eq:Bn}
			\begin{pmatrix}
				-2 &  1 &  0 \\
				1  & -2 &  \eps \\
				0  &  \eps & 2n
			\end{pmatrix}
			\quad\textrm{and}\quad
			\eps = 
			\begin{cases}
				0 & \textrm{if } d = 6n, \\
				1 & \textrm{if } d = 6n + 2.
			\end{cases}
		\end{equation}
	\end{proposition}

	\begin{proof}
		Let $K_d^0 = \Lambda_0\cap K_d$. The lattices $K_d^\perp \subseteq \Lambda$ and $(K_d^0)^\perp\subseteq \Lambda_0$ coincide; we compute the latter. Recall that $\Lambda_0 \isom B\oplus U^{\oplus 2} \oplus E_8^{\oplus 2}$, where $B = \langle a,b\rangle$ has intersection pairing given by~\eqref{eq:B}. Let $\langle e,f\rangle$ be a basis for one of the hyperbolic planes $U$ in $\Lambda_0$, with $e^2 = f^2 = 0$ and $ef = 1$. By~\cite[Proposition~3.2.5 and its proof]{HassettThesis}, a primitive vector in $\Lambda_0$ is congruent under $\Gamma^+_d$ to $e + nf$ if $d = 6n$ or $\pm(a+b) + 3(e + nf)$ if $d = 6n +2$. Applying this result to a primitive vector generating $K_d^0$, we obtain
		\[
			(K_d^0)^\perp \isom \langle e - nf\rangle \oplus B \oplus U \oplus E_8^{\oplus 2},
		\]
in the case when $d = 6n$. In the case when $d = 6n+2$, we get
		\[
			(K_d^0)^\perp \isom \langle \pm(a+b) + 3(e + nf) \rangle^\perp_{B\oplus U} \oplus U \oplus E_8^{\oplus 2}.
		\]
A calculation shows that $\{a-b,\mp a + f,-e + nf\}$ is a basis for the orthogonal complement of $\pm(a+b) + 3(e + nf)$ in $B \oplus U$, with Gram matrix
		\[
			\begin{pmatrix}
				2 & \mp 1 &  0 \\
				\mp 1  & 2 &  {-1} \\
				0  &  {-1} & {-2n}
			\end{pmatrix}
		\]
These two matrices, one for each choice of sign, are conjugate to each other. Take the matrix with value $-1$ in the $(2,1)$ and $(1,2)$ entries. To compute $K_d^\perp(-1)$ we multiply the resulting matrix entries by $-1$, finishing the proof of the proposition.
	\end{proof}


\section{Lattice Engineering}
	\label{S:latticeEng}


	In this section, for large enough $d$, we begin the construction of primitive embeddings of the lattice $K_d^\perp(-1)$ into the Borcherds lattice $L_{2,26}$, in such a way that the set
	\[
		R_{-2}\left( K_d^\perp(-1) \right) := \left\{r \in L_{2,26} : r^2 = -2, \left(r,K_d^\perp(-1)\right) = 0\right\},
	\]
is not empty and has $< 14$ elements, with a view towards applying Theorem~\ref{thm:quasipullback} to construct a low-weight cusp form with respect to the group $\widetilde{O}^+(K_d^\perp(-1))$.

	By Proposition~\ref{prop:Kdperp}, we have $K_d^\perp(-1) \isom B_n \oplus U \oplus E_8(-1)^{\oplus 2}$. Let $\{a_1,a_2,\ell\}$ be an ordered basis for the rank $3$ lattice $B_n$, with intersection pairing as in~\eqref{eq:Bn}. We embed $K_d^\perp(-1)\hookrightarrow L_{2,26}$ by taking the summands $U$ and $E_8(-1)^{\oplus 2}$ identically\footnote{Note that by~\cite[Theorem~1.14.4]{Nikulin}, there is a unique primitive embedding $U\oplus E_8(-1)^{\oplus 2} \hookrightarrow L_{2,26}$ up to isometries, so we may as well take the identity map on the components of $U\oplus E_8(-1)^{\oplus 2}$.} into corresponding summands of $L_{2,26}$, and by primitively embedding $B_n$ into the remaining $U \oplus E_8(-1)$ as follows. First, choose a sublattice $A_2(-1)$ inside $E_8(-1)$ and map $\langle a_1,a_2\rangle$ onto this copy of $A_2(-1)$ in the obvious way\footnote{Again, by~\cite[Theorem~1.14.4]{Nikulin} the embedding $A_2(-1) \hookrightarrow U\oplus E_8(-1)$ is unique up to isometries.}.   It then remains to specify the image of $\ell\in U \oplus E_8(-1)$ in such a way that the set 
	\[
		R_\ell := \left\{r \in U \oplus A_2(-1)^\perp_{E_8(-1)} : r^2 = -2, (r,\ell) = 0\right\}
	\]
is not empty and has $< 14$ elements. 

	\begin{lemma} 
		\label{lem:types}
		Let $\ell = \alpha e + \beta f + v$, where $U = \langle e,f\rangle, v \in E_8(-1)$, and $\alpha, \beta \in \Z$. Let $r = \alpha' e + \beta' f + v'$ be a vector of $R_\ell$, where $v' \in A_2(-1)^\perp_{E_8(-1)} \isom E_6(-1)$ and $\alpha',\beta' \in \Z$.  If $n < \alpha\beta < 2n$ then $\alpha'\beta' \in \{-1,0\}$. Moreover, if $\alpha'\beta' = -1$, then $\alpha = \beta$.
	\end{lemma}

	\begin{proof}
		By definition of $R_\ell$, we have
		\begin{align}
			\label{eq:length}
			r^2 &= 2\alpha'\beta' + (v')^2 = -2, \textrm{ and} \\
			\label{eq:ortho}
			r\cdot\ell &= \alpha'\beta + \alpha\beta' + v\cdot v' = 0.
		\end{align}
If $\alpha' = 0$ we are done; assume then that $\alpha' \neq 0$. The relation~\eqref{eq:length} implies that
		\[
			\beta' = -\frac{2 + (v')^2}{2\alpha'}.
		\]
Substituting into~\eqref{eq:ortho}, we obtain the quadratic relation in $\alpha'$
		\[
			(2\beta)(\alpha')^2 + (2v\cdot v')\alpha' -\alpha(2 + (v')^2) = 0.
		\]
The discriminant of this quadratic equation must be nonnegative, implying that
		\[
			(v\cdot v')^2 \geq -2\alpha\beta(2 + (v')^2).
		\]
		The Cauchy-Schwarz inequality gives $v^2(v')^2 \geq (v\cdot v')^2$. On the other hand, we know that $\ell^2 = 2n$, so $v^2 = 2(n - \alpha\beta)$. Combining these facts we conclude that
		\[
			2(n - \alpha\beta)(v')^2 \geq -2\alpha\beta(2 + (v')^2), \textrm{ so that } (v')^2 \geq -\frac{2\alpha\beta}{n}.
		\]
		The assumption that $\alpha\beta < 2n$ then gives $(v')^2 > -4$. The vector $v'$ is in $E_6(-1)$, leaving us with the possibilities $(v')^2 = 0$ or $-2$. If $(v')^2 = 0$, then $v' = 0$ because $E_6(-1)$ is negative definite. But then $r^2 = -2$ implies $\alpha'\beta' = -1$, meaning $(\alpha',\beta') = (-1,1)$ or $(1,-1)$.  Either way, $r\cdot\ell = 0$ then implies that $\alpha = \beta$.  If $(v')^2 = -2$ then the condition $r^2 = -2$ forces $\alpha'\beta' =0$.  
	\end{proof}

	\begin{remark}
		\label{rem:types}
		Let $r \in R_\ell$ be a vector as Lemma~\ref{lem:types}, and suppose that $n < \alpha\beta < 2n$, and that $\alpha \neq \beta$, so that, without loss of generality, $r = \alpha' e + v'$ for some $\alpha'\in \Z$. The condition $r\cdot\ell = 0$ implies that $\alpha' = -(v\cdot v')/\beta$.
In particular, $\beta\mid v\cdot v'$. 
	\end{remark}

	Together, Lemma~\ref{lem:types} and Remark~\ref{rem:types} show that if $\alpha \neq \beta$ and $n < \alpha\beta < 2n$, then there are three types of vectors $r = \alpha'e + \beta'f + v' \in R_\ell$:
	\smallskip
	\begin{itemize}
		\item {\bf Type I}: vectors with $\alpha' = \beta' = 0$. In this case $r$ is a root in $E_6(-1)$.
		\smallskip
		\item {\bf Type II}: vectors with $\alpha'\neq 0$ and $\beta' = 0$. In this case $\alpha'\beta + v\cdot v' = 0$, so $v\cdot v' \neq 0$ and $v\cdot v' \equiv 0 \bmod \beta$.
		\smallskip
		\item {\bf Type III}: vectors with $\beta'\neq 0$ and $\alpha' = 0$. In this case $\alpha\beta' + v\cdot v' = 0$, so $v\cdot v' \neq 0$ and $v\cdot v' \equiv 0 \bmod \alpha$.
	\end{itemize}

	\begin{lemma}
		\label{lem:onlyTypeI}
		Suppose that $\alpha \neq \beta$, and that the following inequalities hold:
		\begin{equation}
			\label{eq:bareminimum}
			\sqrt{n} < \alpha,\qquad \sqrt{n} < \beta,\qquad \alpha\beta < 5n/4.
		\end{equation}
		Then there are no Type II or Type III vectors in $R_\ell$.
	\end{lemma}

	\begin{proof}
		Lemma~\ref{lem:types} and its proof imply that if $r = \alpha'e + \beta'f + v' \in R_\ell$, then $\alpha' = 0$ or $\beta' = 0$, and that $(v')^2 = -2$.  Applying the Cauchy-Schwarz inequality we obtain
		\[
			|v\cdot v'| \leq \sqrt{v^2\cdot(v')^2} = \sqrt{2(n - \alpha\beta)(-2)} < \sqrt{n},
		\]
where the last inequality follows from $\alpha\beta < 5n/4$.  If $r$ is a Type II or Type III vector, then either $\alpha$ or $\beta$ divides $v\cdot v'$, and $v\cdot v' \neq 0$, so that $\sqrt{n} < |v\cdot v'|$. This is a contradiction.
\end{proof}

	\begin{remark}
		In the sequel, we ensure that the inequalities~\eqref{eq:bareminimum} hold by assuming there is an $\epsilon > 0$ such that the inequalities
		\begin{equation}
			\label{eq:notypeIIorIII}
			\sqrt{(1 + \epsilon)n} < \alpha < \sqrt{5n/4},\qquad
			\sqrt{(1 + \epsilon)n} < \beta < \sqrt{5n/4}
		\end{equation}
		hold.  An $\epsilon > 0$ gives us added flexibility when trying to satisfy additional conditions on top of~\eqref{eq:bareminimum} when $n \gg 0$, like $3(\alpha\beta - n) - 1 > 20$; see, for example, the proof of Theorem~\ref{thm:reduction}.
	\end{remark}


\section{The case $d \equiv 2 \bmod 6$}
	\label{s:2mod6}


\subsection{Preliminary reductions}
	\label{ss:PrelimReds}

	Retain the notation of \S\ref{S:latticeEng}. Let $d = 6n + 2$ for some positive integer $n$, and assume first that $n$ is sufficiently large so that there exist integers $\alpha$ and $\beta$ satisfying the hypotheses of Lemma~\ref{lem:onlyTypeI}. Then any $r \in R_\ell$ must be a Type I vector, and thus is a root in $E_6(-1)$.  This reduces the problem of finding a vector $\ell$ to finding a vector $v \in E_8(-1)$ of length $v^2 = 2(n - \alpha\beta)$ with $(v,a_1) = 0$ and $(v,a_2) = 1$ (where $A_2(-1) = \langle a_1,a_2\rangle$), such that
	\[
		0 < \#\{r \in E_6(-1) : r^2 = -2, (r,v) = 0\} < 14.
	\]
We have inclusions of abelian groups (in fact, of $\Q$-lattices)
	\begin{equation}
		\label{eq:inclusions}
		E_6(-1)\oplus A_2(-1) \subset E_8(-1) = E_8(-1)^\vee \subset E_6(-1)^\vee\oplus A_2(-1)^\vee.
	\end{equation}
We write $v = v_1^\vee + v_2^\vee$ with $v_1^\vee \in E_6(-1)^\vee$ and $v_2^\vee \in A_2(-1)^\vee$. The conditions $(v,a_1) = 0$ and $(v,a_2) = 1$ imply that $v_2^\vee = -(a_1+2a_2)/3$. Since $(v_2^\vee)^2 = -2/3$ we can further reduce our problem to finding a vector $v_1^\vee \in E_6(-1)^\vee$ of length $(v_1^\vee)^2 = 2(n - \alpha\beta) + 2/3$, such that
	\[
		v_1^\vee + v_2^\vee \in E_8(-1)\quad\textrm{and}\quad 0 < \#\{r \in E_6(-1) : r^2 = -2, (r,v_1^\vee) = 0\} < 14.
	\]
For later notational convenience, we rephrase this equivalent problem using positive definite lattices $E_6$ and $E_8$, as follows.

	\begin{problem}
		\label{prob:lattices}
		Let $n$ be an integer that is sufficiently large so that there exist $\epsilon > 0$ and coprime integers $\alpha$ and $\beta$ satisfying the inequalities~\eqref{eq:notypeIIorIII}. Let $\langle a_1, a_2\rangle$ be a sublattice of type $A_2$ in $E_8$ with orthogonal complement $E_6$, and let $v_2^\vee = -(a_1+2a_2)/3 \in A_2^\vee$. Find a vector $v_1^\vee \in E_6^\vee$ of length $(v_1^\vee)^2 = 2(\alpha\beta - n) - 2/3$ such that
		\[
			v_1^\vee + v_2^\vee \in E_8\quad\textrm{and}\quad 0 < \#\{r \in E_6 : r^2 = 2, (r,v_1^\vee) = 0\} < 14.
		\]
	\end{problem}

	The positive definite version of the inclusions~\eqref{eq:inclusions} reads
	\[
		E_6\oplus A_2 \subset E_8 = E_8^\vee \subset E_6^\vee\oplus A_2^\vee.
	\]
	Let $[0]$, $[1]$ and $[2]$ denote representatives for the elements of the discriminant group $D(E_6) \isom \Z/3\Z$. Let $T = E_6 + [1]$ and $T' = E_6 + [2]$, so that $E_6^\vee = E_6 \sqcup T \sqcup T'$ (disjoint union). Without loss of generality we may assume that 
	\begin{enumerate}
		\item $[0] = 0 \in E_6$, i.e., $[0]$ is the identity element for the underlying abelian group of $E_6$,
		\smallskip
		\item $[1] + v_2^\vee \in E_8$, and
		\smallskip
		\item $[2] = {-[1]}$.
	\end{enumerate}
	Condition $(2)$ implies that $v_1^\vee + v_2^\vee \in E_8$ is equivalent to $v_1^\vee \in T$. Let $a \in E_6$ be a root, and let $M_1 = \langle a\rangle^\perp_{E_6^\vee}$.  Abstractly, as a $\Q$-lattice, $M_1$ does not depend on the choice of the root $a$, since the Weyl group $W(E_6)$ acts transitively on the set of roots $R(E_6)$ of $E_6$. To solve Problem~\ref{prob:lattices}, it suffices to find $v_1^\vee \in M_1\cap T$ of length $2(\alpha\beta - n) - 2/3$ such that 
	\[
		\#\{r \in E_6 : r^2 = 2, (r,v_1^\vee) = 0\} < 14.
	\]
To study this problem, we separate the roots $R(E_6)$ of the lattice $E_6$ as follows:
	\[
		R(E_6) = R\left(\langle a\rangle^\perp_{E_6}\right) \sqcup X_{42},\quad\textrm{where }X_{42} = \{c \in R(E_6) : (a,c)\neq 0\}. 
	\]
The proof of the following lemma mirrors that of~\cite[Lemma~7.2(i)]{GHSInventiones}; we include it for completeness.

	\begin{lemma}
		\label{lem:roots of E6}
		The set $X_{42}$ is the union of 10 root systems of type $A_2$ such that 
		\[
			R\left(A_2^{(i)}\right) \cap R\left(A_2^{(j)}\right) = \{\pm a\} \textrm{ for any }i \neq j.
		\]
	\end{lemma}

	\begin{proof}
		By the Cauchy-Schwarz inequality, any two roots $b$, $c$ of $E_6$ satisfy
		\[
			|(b,c)| \leq \sqrt{b^2\cdot c^2} = 2,
		\]
with equality if and only if $b = \pm c$. Assume that $(a,c) = -1$ (otherwise replace $c$ with ${-c}$), so that $A_2(a,c) := \Z a + \Z c$ is a lattice of type $A_2$. This lattice contains six roots:
		\[
			R\left(A_2(a,c)\right) = \{\pm a, \pm c, \pm (a + c)\},
		\]
and it is generated by any pair of linearly independent roots. Hence
		\[
			R\left(A_2(a,c_1)\right) \cap R\left(A_2(a,c_2)\right) = \{\pm a\}.
			\tag*{\qed}
		\]
		\hideqed
	\end{proof}

	Assume that the set $\{r \in E_6 : r^2 = 2, (r,v_1^\vee) = 0\}$ has at least 14 elements.  The vector $v_1^\vee$ must be orthogonal to some roots in $R\left(\langle a\rangle^\perp_{E_6}\right)$ (30 roots) or in $X_{42}\setminus\{\pm a\}$ (40 roots). From Lemma~\ref{lem:roots of E6}, it follows that
	\begin{equation}
		\label{eq:decomp}
		v_1^\vee \in \bigcup_{i = 1}^{10}\left( \left(A_2^{(i)}\right)^\perp_{E_6^\vee} \cap T\right)
		\cup \bigcup_{i = 1}^{15}\left( \left(A_1^{(i)}\right)^\perp_{M_1}\cap T \right).
	\end{equation}
	Define
	\[
		M_2 = \left( A_2\right)^\perp_{E_6^\vee}
		\quad\textrm{and}\quad
		M_3 = \left(A_1\right)^\perp_{M_1} = \left(A_1\oplus A_1\right)^\perp_{E_6^\vee},
	\]
and note that $M_2$ does not depend on the particular choice of $A_2$ in $E_6 \subseteq E_6^\vee$, because the Weyl group $W(E_6)$ acts transitively on the set of $A_2$-sublattices of $E_6$. Similarly, $M_3$ does not depend on the particular choice of $A_1\oplus A_1$ in $E_6 \subseteq E_6^\vee$. Let $\mathfrak{n}(v_1^\vee)$ be the number of components in the decomposition~\eqref{eq:decomp} that contain $v_1^\vee$.

	\begin{lemma}
		\label{lem:numcomps}
		Let $v_1^\vee \in M_1\cap T$ be a vector of length $2(\alpha\beta - n) - 2/3$. Assume that the set $R\left(\langle v_1^\vee\rangle^\perp_{E_6}\right) = \{r \in E_6 : r^2 = 2, (r,v_1^\vee) = 0\}$ has at least 14 elements.  Then $\frakn(v_1^\vee)\geq 4$.
	\end{lemma}

	\begin{proof}
		We perform an analysis similar to~\cite[p.\ 555]{GHSInventiones}. Suppose first that $(v_1^\vee,c) \neq 0$ for any $c \in X_{42}\setminus\{\pm a\}$. Then in the decomposition~\eqref{eq:decomp}, the vector $v_1^\vee$ can belong only to components of the form $M_3\cap T$. Since the set $R\left(\langle v_1^\vee\rangle^\perp_{E_6}\right)$ contains $\{\pm a\}$, and at least $12$ other elements, by hypothesis, we conclude that $\frakn(v_1^\vee)\geq 6$.

		Suppose now that there exists a $c \in X_{42}\setminus\{\pm a\}$ such that $(v_1^\vee,c) = 0$. Then $v_1^\vee$ is orthogonal to $R\left(A_2(a,c)\right)\subseteq X_{42}$, i.e., $v_1^\vee$ belongs to at least one component of the form $M_2\cap T$ in the decomposition~\eqref{eq:decomp}. 
		\smallskip
		\begin{enumerate}
			\item If no other components of this type contain $v_1^\vee$, then since $\#R\left(\langle v_1^\vee\rangle^\perp_{E_6}\right) > 14$, we know that $v_1^\vee$ lies in at least $4$ components of the form $M_3\cap T$, so that $\frakn(v_1^\vee) \geq 5$. 
			\smallskip
			\item If $v_1^\vee$ is contained in exactly two components of the form $M_2\cap T$, then it is orthogonal to the roots of the root system $A_2^{(i)} + A_2^{(j)}$; this system is isomorphic to $A_3$ (see~\cite[Lemma~7.2(ii)]{GHSInventiones}); thus $v_1^\vee$ is orthogonal to the $12$ roots of this $A_3$.  Since $\#R\left(\langle v_1^\vee\rangle^\perp_{E_6}\right) > 14$, we conclude that $v_1^\vee$ is contained in at least one component of the form $M_3\cap T$, and hence $\frakn(v_1^\vee) \geq 4$ in this case. 
			\smallskip
			\item If $v_1^\vee$ is contained in exactly three components of the form $M_2\cap T$, then it is orthogonal to the roots in the system $A_2^{(i)} + A_2^{(j)} + A_2^{(k)}$ which is isomorphic to either $A_4$ or $D_4$; in either case there are three components of the form $M_3 \cap T$ that contain $v_1^\vee$.  See~\cite[Lemma~7.2(iii)]{GHSInventiones} for details of this calculation, taking into account that the $E_7^{(a)}$ in their notation can be replaced by $\langle a\rangle^\perp_{E_6}$ in our notation without so changing their proof. We conclude that $\frakn(v_1^\vee) \geq 6$.
			\smallskip
			\item If $v_1^\vee$ is contained in at least four components of the form $M_2\cap T$, then we are already done. \qed
		\end{enumerate}
		\hideqed
	\end{proof}

	\begin{theorem}
		\label{thm:reduction}
		Let $n$ be an integer that is sufficiently large so that there exist $\epsilon > 0$ and coprime integers $\alpha$ and $\beta$ satisfying the inequalities~\eqref{eq:notypeIIorIII}. Let $m := 2(\alpha\beta - n) - 2/3$. Suppose that
		\begin{equation}
			\label{eq:latticeineq}
			10N_{M_2\cap T}\left(m\right) + 15N_{M_3\cap T}\left(m\right) < 4N_{M_1 \cap T}\left(m\right)
		\end{equation}
Then there exists a vector $v_1^\vee \in M_1 \cap T \subseteq E_6^\vee$ of length $m$ such that 
		\[
			0 < \#R\left( \langle v_1^\vee\rangle^\perp_{E_6}\right) < 14.
		\]
		The vector $v_1^\vee$ yields a solution to Problem~\ref{prob:lattices}.
	\end{theorem}

	\begin{proof}
		Suppose, to the contrary, that every vector $v_1^\vee \in M_1\cap T$ of length $m$ is orthogonal to at least $14$ roots of $E_6$.  On the one hand, by the decomposition~\eqref{eq:decomp} and definition of $\frakn(v_1^\vee)$, we have
		\[
			10N_{M_2\cap T}\left(m\right) + 15N_{M_3\cap T}\left(m\right) = \sum_{\substack{v_1^\vee \in M_1 \cap T \\ (v_1^\vee)^2 = m}} \frakn(v_1^\vee).
		\]
On the other hand, Lemma~\ref{lem:numcomps} implies that
		\[
			\sum_{\substack{v_1^\vee \in M_1 \cap T \\ (v_1^\vee)^2 = m}} \frakn(v_1^\vee)
			\geq 4N_{M_1 \cap T}\left(m\right),
		\]
contradicting the assumed inequality~\eqref{eq:latticeineq}.
	\end{proof}


\subsection{Explicit realizations of $M_i$, $i = 1$, $2$ and $3$}

	To apply Theorem~\ref{thm:reduction} to the problem of constructing embeddings $K_d^\perp(-1) \hookrightarrow L_{2,26}$ with only a few roots of $L_{2,26}$ orthogonal to the image, we need a lower bound for $N_{M_1\cap T}\left(m\right)$, as well as upper bounds for $N_{M_2\cap T}\left(m\right)$ and $N_{M_3\cap T}\left(m\right)$. To this end, we fix explicit realizations of the lattices $M_i$, $i = 1$, $2$ and $3$. 

	Following the conventions in~\cite[p.\ 127]{CS}, we let $e_1,\dots,e_8$ denote the standard basis of $\Q^8$, considered as a lattice with the standard dot product. Then
	\begin{align*}
		E_6^\vee &= \langle e_3 - e_2, e_4 - e_3, e_5 - e_4, e_6 - e_5, (2e_2 + 2e_3 - e_4 - e_5 - e_6 - e_7)/3, \\
		&\quad\ (e_1 + e_2 + e_3 + e_4 - e_5 - e_6 - e_7 - e_8)/2 \rangle,
	\intertext{The vector $a = e_3 - e_2$ is a root of $E_6 \subseteq E_6^\vee$. Hence, in these coordinates,}
		M_1 &= \langle a\rangle^\perp_{E_6^\vee} = \langle 2e_4 - e_3 - e_2, e_5 - e_4, e_6 - e_5, (2e_2 + 2e_3 - e_4 - e_5 - e_6 - e_7)/3, \\
		&\qquad\qquad\quad (e_1 + e_2 + e_3 + e_4 - e_5 - e_6 - e_7 - e_8)/2 \rangle.
	\intertext{We have $\left\langle e_3 - e_2,(e_1 + e_2 - e_3 + e_4 + e_5 - e_6 - e_7 - e_8)/2\right\rangle \isom A_2 \subseteq E_6$. Hence}
		M_2 &= (A_2)^\perp_{E_6^\vee} = \langle -e_2 - e_3 + 2e_4 - e_5 + e_6, -e_4 + e_5, \\
		&\qquad\qquad\qquad (e_1 + e_2 + e_3 + e_4 - 3e_5 + e_6 - e_7 - e_8)/2, \\
		&\qquad\qquad\qquad (2e_2 + 2e_3 - e_4 - e_5 - e_6 - e_7)/3\rangle
	\intertext{Finally, the vector $e_6-e_5$ is a root in $M_1$, and so}
		M_3 &= (A_1)_{M_1}^\perp = \langle -2e_4 + e_5+e_6,-e_2-e_3+2e_4, \\
		&\qquad\qquad\qquad (2e_2+2e_3-e_4-e_5-e_6-e_7)/3, \\
		&\qquad\qquad\qquad (e_1+e_2+e_3+e_4-e_5-e_6-e_7-e_8)/2\rangle
	\end{align*}
Let $B_1$, $B_2$ and $B_3$ be the matrices
	\[
		\begin{bmatrix}
			 6 & -2 &  0 & -2 &  0 \\
			-2 &  2 & -1 &  0 & -1 \\ 
			 0 & -1 &  2 &  0 &  0 \\
			-2 &  0 &  0 &4/3 &  1 \\
			 0 & -1 &  0 &  1 &  2
		\end{bmatrix},
		\quad
		\begin{bmatrix}
			 8 & -3 &  2 & -2 \\
			-3 &  2 & -2 &  0 \\
			 2 & -2 &  4 &  1 \\
			-2 &  0 &  1 & 4/3 
		\end{bmatrix},
		\quad\text{and}\quad
		\begin{bmatrix}
		 6 & -4 &  0  & -2 \\
		-4 &  6 & -2  &  0 \\
		 0 & -2 & 4/3 &  1 \\
		-2 &  0 &  1  &  2 
		\end{bmatrix},
	\]
of determinants $2/3$, $1$ and $4/3$, respectively. Then $B_i$ is the Gram matrix for $M_i$ under the above bases ($i = 1$, $2$ and $3$).
The matrices $3B_i$ are symmetric, even, integral, and positive definite; we set 
	\[
		S_i = S_i(x) = \frac{1}{2}x^t(3B_i)\,x\quad i = 1, 2, 3.
	\]
A {\tt Magma} computation shows that each of the three quadratic forms $S_i$ is unique in its genus (see the script {\tt ThetaFunctionComputations} included in the {\tt arXiv} distribution of this article). We shall use this fact to apply Theorem~\ref{thm:representationsbyintegers} to these forms.


\subsection{Lower bound for $N_{M_1\cap T}\left(m\right)$}
\label{ss:M1Bound}

	\begin{prop}
		\label{cor:M1Bound}
		Let $t$ be an integer congruent to $2$ or $11 \bmod 12$, and let $m = 2t/3$. Then 
		\[
			N_{M_1\cap T}\left(m\right) > 5.2488\cdot m^{3/2}.
		\]
	\end{prop}

	\begin{proof}
		Recall that $M_1 = \langle a\rangle^\perp_{E_6^\vee}$, where $a$ is a root of $E_6$, that $T = E_6 + [1]$, and that $T' = E_6 + [2]$. Hence $M_1 \cap E_6 \isom A_5$, and there is a disjoint union decomposition
		\[
			M_1 = A_5 \sqcup (M_1\cap T) \sqcup (M_1 \cap T').
		\]
		In terms of the coefficients of the associated theta series of these lattices, we obtain
		\[
			N_{M_1}(m) = N_{A_5}(m) + N_{M_1\cap T}(m) + N_{M_1 \cap T'}(m).
		\]
		On the other hand, the map $v\mapsto {-v}$ gives a length-preserving bijection between the lattice cosets $M_1\cap T$ and $M_1 \cap T'$, and hence
		\begin{equation*}
			2N_{M_1 \cap T}(m) = N_{M_1 \cap T}(m) +N_{M_1 \cap T'}(m)
			= N_{M_1}(m) - N_{A_5}(m).
		\end{equation*}
		Since $m$ is not an integer, we have $N_{A_5}(m) = 0$; hence estimating $N_{M_1\cap T}\left(m\right)$ and $N_{M_1}\left(m\right)/2$ are equivalent problems, and
		\begin{equation*}
			\label{eq:Ntor}
			N_{M_1\cap T}\left(m\right) = \frac{N_{M_1}\left(m\right)}{2} = \frac{r\left(3m/2,S_1\right)}{2}.
		\end{equation*}
		Applying Theorem~\ref{thm:representationsbyintegers} to $S = S_1$ and $t = 3m/2$ we obtain
		\begin{equation}
			\label{eq:M1capT}
			\begin{split}
				N_{M_1\cap T}\left(m\right) &= \frac{1}{2}(2\pi)^{5/2}\Gamma(5/2)^{-1}\cdot t^{3/2}\cdot |2\cdot 3^4|^{-1/2}L(2,Dt_2^2)\zeta(4)^{-1}\\
				&\qquad \times\prod_{p = 2, 3}\frac{1 - \chi_D(p)p^{-2}}{1 - p^{-4}}\alpha_p\left(t,S_1\right),
			\end{split}
		\end{equation}
where $D = \disc \Q(\sqrt{t})$ and $\chi_D(a) = \left(\frac{D}{a}\right)$ is the Jacobi symbol for $a \in \Z_{>0}$.
To bound $N_{M_1\cap T}\left(m\right)$ from below, we first bound $L(2,Dt_2^2)$. By definition we have
		\begin{align}
			\sum_{n \geq 1} \frac{b_n(Dt_2^2)}{n^s} &= \frac{\zeta(s)L(s, Dt_2^2)}{\zeta(2s)} \tag*{}\\
		\intertext{The chinese reminder theorem implies that $b_n(Dt_2^2)$ is multiplicative in $n$, giving an Euler product expansion}
			\sum_{n \geq 1} \frac{b_n(Dt_2^2)}{n^s} &= \prod_p \left( \sum_{k\geq 0} \frac{b_{p^k}(Dt_2^2)}{p^{ks}}\right). \tag*{}\\
		\intertext{All local factors at odd primes $p$ are $\geq 1$; the assumption that $t \equiv 2$ or $3 \bmod 4$ implies that $b_2(Dt_2^2)= 1$, giving a lower bound of $1+1/2^s$ for the local factor of $p=2$. Hence}
\sum_{n \geq 1} \frac{b_n(Dt_2^2)}{n^s} &\geq 1+\frac{1}{2^s}, \tag*{}\\
		\intertext{and consequently, we have the estimate}
			\label{lem:Lbound}
			L(2, Dt_2^2) &\geq \frac{\zeta(4)}{\zeta(2)}\cdot \frac{5}{4}.
		\end{align}

		We turn to the local densities $\alpha_2(t,S_1)$ and $\alpha_3(t,S_1)$, and claim that
		\begin{equation}
			\label{lem:AlphaBound}
			\alpha_2(t,S_1) \geq \frac{25}{28}\quad\textrm{and}\quad \alpha_3(t,S_1) = 2.
		\end{equation}
To see this, first note that the matrix $(3/2)B_1$ is $\Z_2$-equivalent to
		\[
			\begin{pmatrix}
			1
			\end{pmatrix}
			\oplus
			\begin{pmatrix}
				1/3 & 1/6 \\
				1/6 & 1/3
			\end{pmatrix}
			\oplus
			\begin{pmatrix}
				1/3 & 1/6 \\
				1/6 & 1/3
			\end{pmatrix}
		\]
so in the notation of~\eqref{eq:diagZ2}, we have $\ell_1 = n_1 = n_2 = 0$, $\epsilon_1 = 1$, and $\epsilon_1'' = \epsilon_2'' = 1/3$. Since
		\[
			\alpha_2(t,S_1) \geq 1 - |R_1(t,S_1)|
		\]
we need only obtain an upper bound for $|R_1(t,S_1)|$. For this, we assume that the quantities $\delta(k)$, $p(k)$, $\big(\tfrac{2}{\mu\epsilon(k)}\big)$, $e^{-2\pi i(\mu/8)}$, and $\Char(4\Z_2)(\mu)$
 in~\eqref{eq:R1Z2} are all equal to $1$.  We also note that $\ell(k-1,1)$ being odd is equivalent to $k$ being even, while $\ell(k-1,1)$ is even precisely when $k$ is odd.  Using~\eqref{eq:diagZ2} we compute that $d(k) = -3k/2 + 1/2$ except if $k = 1$ (in which case $\delta(k) = 0$, so we ignore these terms). Putting all of this together, we get
		\[
			|R_1(t,S_1)| \leq \sum_{\substack{1 < k \\ k \textrm{ even}}} 2^{-3k/2 - 1} + 
			\sum_{\substack{1 < k \\ k \textrm{ odd}}} 2^{-3k/2 - 1/2} = \frac{3}{28},
		\]
and hence
		\[
			\alpha_2(t,S_1) \geq 1 - \frac{3}{28} = \frac{25}{28}.
		\]
Over $\Z_3$, the matrix $(3/2)B_1$ of $S_1$ can be diagonalized to $\diag(2,6,6,6,3)$. In the notation of~\eqref{eq:ZpNormalization} we have $\ell_1 = 0$, $\ell_2 = \ell_3 = \ell_4 = \ell_5 = 1$ and $\epsilon_1=\epsilon_2=\epsilon_3=\epsilon_4=2$, $\epsilon_5 = 1$. Since $t = \frac{3}{2}m(\alpha,\beta,n)$ is not divisible by $3$, the quantity $a$ in~\eqref{eq:R1pOdd} is zero, and thus $R_1(t,S_1) = v(1)\cdot 3^{d(1)}\cdot f(t) = 1$. We conclude that $\alpha_3(t,S_1) = 2$, establishing~\eqref{lem:AlphaBound}.

		To finish the proof, we apply the inequality
		\[
			\frac{1 - \chi_D(p)p^{-2}}{1 - p^{-4}} \geq \frac{1 - p^{-2}}{1 - p^{-4}},
		\]
for $p = 2$, $3$, as well as \eqref{lem:Lbound} and~\eqref{lem:AlphaBound} to~\eqref{eq:M1capT}.
	\end{proof}


\subsection{Upper bounds for $N_{M_2\cap T}\left(m\right)$ and $N_{M_3\cap T}\left(m\right)$}

	\begin{prop}
		\label{cor:M2Bound}
		Let $t$ be an integer congruent to $2 \bmod 3$, and let $m = 2t/3$. Then 
		\begin{align*}
			N_{M_2\cap T}\left(m\right) &< \frac{9m}{4} \cdot\left(\ln\left(3m/2\right) + 1\right), \textrm{ and}\\
			N_{M_3\cap T}\left(m\right) &< 9.0004 \cdot m\cdot\left(\ln\left(3m/2\right) + 1\right).
		\end{align*}
	\end{prop}

	\begin{proof}
		Let us first consider the upper bound for $N_{M_2\cap T}\left(m\right)$.
Recall that after fixing $A_2 \subseteq E_6$, we put $M_2 = (A_2)^\perp_{E_6^\vee}$, so that $M_2 \cap E_6 = (A_2)^\perp_{E_6} \isom A_2\oplus A_2$. The
disjoint union decomposition
		\[
			M_2 = (A_2\oplus A_2) \sqcup (M_2\cap T) \sqcup (M_2 \cap T'),
		\]
where $T = E_6 + [1]$, and $T' = E_6 + [2]$, gives an equality of coefficients of associated theta series:
		\[
			N_{M_2}(m) = N_{A_2\oplus A_2}(m) + N_{M_2\cap T}(m) + N_{M_2 \cap T'}(m).
		\]
As in the proof of Proposition~\ref{cor:M1Bound}, we have $N_{M_2 \cap T}(m) = N_{M_2 \cap T'}(m)$ and $N_{A_2\oplus A_2}(m) = 0$ because $m$ is not an integer, so that estimating $N_{M_2\cap T}\left(m\right)$ and $N_{M_2}\left(m\right)/2$ are equivalent problems, and
		\begin{equation*}
			\label{eq:Ntor2}
			N_{M_2\cap T}\left(m\right) = \frac{N_{M_2}\left(m\right)}{2} = \frac{r\left(3m/2,S_2\right)}{2}
		\end{equation*}
Applying Theorem~\ref{thm:representationsbyintegers} to $S = S_2$ and $t = 3m/2$, we obtain 
		\begin{equation}
			\label{eq:M2capT}
			\begin{split}
				N_{M_2\cap T}\left(m\right) &= \frac{1}{2}\bigg(\sum_{a\mid t}\chi_{4D}(a)a^{-1}\bigg)L(2,\chi_{4D})^{-1} \\
				&\qquad \times (2\pi)^2\cdot\Gamma(2)^{-1}\cdot t\cdot |3^4|^{-1/2} \prod_{p \mid 2D}\alpha_p(t,S_2),
			\end{split}
		\end{equation}
The $L$-function term can be computed via its Euler product expansion
		\begin{equation}
			\label{eq:Lvalue}
			L(2,\chi_{4D}) = \prod_p \left(1 - \frac{\chi_{4D}(p)}{p^2}\right)^{-1} = \prod_{p \neq 2, 3} \left(1 - \frac{1}{p^2}\right)^{-1} = \left(1 - \frac{1}{2^2}\right)\cdot \left(1 - \frac{1}{3^2}\right)\cdot \frac{\pi^2}{6}.
		\end{equation}
Next, note that
		\begin{equation}
			\label{eq:logisin}
			\sum_{a\mid t}\chi_{4D}(a)a^{-1} \leq \sum_{a\mid t}\frac{1}{a} \leq \sum_{a = 1}^t \frac{1}{a} < \log t + 1.
		\end{equation}

		We turn to the local densities $\alpha_2(t,S_1)$ and $\alpha_3(t,S_1)$, and claim that
		\begin{equation}
			\label{lem:M2locDensities}
			\alpha_2(t,S_2) \leq \frac{3}{2}\quad\textrm{and}\quad \alpha_3(t,S_2) = 2.
		\end{equation}
To see this, note that over $\Z_2$ the matrix $(3/2)B_2$ is equivalent to
		\[
			\begin{pmatrix}
				1/3 & 1/6 \\
				1/6 & 1/3
			\end{pmatrix}
			\oplus
			\begin{pmatrix}
				1/3 & 1/6 \\
				1/6 & 1/3
			\end{pmatrix}
		\]
so in the notation of~\eqref{eq:diagZ2}, we have $n_1 = n_2 = 0$ and $\epsilon_1'' = \epsilon_2'' = 1/3$. Since
		\[
			\alpha_2(t,S_2) \leq 1 + |R_1(t,S_2)|,
		\]
we look for an upper bound for $|R_1(t,S_2)|$.  We assume that the quantities $\delta(k)$, $p(k)$, $\big(\frac{2}{\mu\epsilon(k)}\big)$, $e^{-2\pi i(\mu/8)}$, and $\Char(4\Z_2)(\mu)$
 in~\eqref{eq:R1Z2} are all equal to $1$.  The quantity $\ell(k-1,1)$ is always zero, therefore even. We compute that $d(k) = -k$, and thus obtain
		\[
			|R_1(t,S_2)| \leq \sum_{0 < k < \infty} 2^{-k-1} = \frac{1}{2}.
		\]
Hence $\alpha_2(t,S_2) \leq 1 + \frac{1}{2} = \frac{3}{2}$,
as claimed.  The proof that $\alpha_3(t,S_2) = 2$ is similar to the proof that $\alpha_3(t,S_1) = 2$.

		Finally, applying~\eqref{eq:Lvalue},~\eqref{eq:logisin}, and~\eqref{lem:M2locDensities} to~\eqref{eq:M2capT}, we obtain the desired bound on $N_{M_2\cap T}(m)$.

		For the lattice $M_3\cap T$, we have
		\begin{equation*}
			\label{eq:Ntor3}
			N_{M_3\cap T}\left(m\right) = \frac{N_{M_3}\left(m\right)}{2} = \frac{r\left(3m/2,S_3\right)}{2},
		\end{equation*}
so that, applying Theorem~\ref{thm:representationsbyintegers} to $S = S_3$ and $t = 3m/2$, we obtain 
		\begin{equation}
			\label{eq:M3capT}
			\begin{split}
				N_{M_3\cap T}\left(m\right) &= \frac{1}{2}\bigg(\sum_{a\mid t}\chi_{4D}(a)a^{-1}\bigg)L(2,\chi_{4D})^{-1} \\
				&\qquad \times (2\pi)^2\cdot\Gamma(2)^{-1}\cdot t\cdot |3^4|^{-1/2} \prod_{p \mid 2D}\alpha_p(t,S_2),
			\end{split}
		\end{equation}
For the $L$-function term, we compute
		\begin{align*}
			L(2,\chi_{4D}) &= \left(1 - \frac{1}{2^2}\right)L(2,\chi_{3}) = \frac{3}{4}\sum_{n=1}^\infty\frac{\chi_{3}(n)}{n^2} \\
			&\geq \frac{3}{4}\left(\sum_{n=1}^k\frac{\chi_{3}(n)}{n^2} - \sum_{n=k+1}^\infty\frac{1}{n^2}\right) \\
			&= \frac{3}{4}\left(-\sum_{n=1}^\infty\frac{1}{n^2} + \sum_{n=1}^k\frac{\chi_3(n) + 1}{n^2}\right).
		\end{align*}
Putting $k = 25,000$ in a computer algebra system, using the equality $\sum (1/n^2) = \pi^2/6$ and taking inverses gives
		\begin{equation}
			\label{lem:LfunctionBoundM3}
			L(2,\chi_{4D})^{-1} < 1.0530.
		\end{equation}
For the local density terms, one computes, as in the case of $S_1$ and $S_2$, that
		\begin{equation}
			\label{lem:M3locDensities}
			\alpha_2(t,S_3) \leq \frac{3}{2}\quad\textrm{and}\quad \alpha_3(t,S_3) = 2.
		\end{equation}
Applying~\eqref{lem:LfunctionBoundM3},~\eqref{eq:logisin}, and~\eqref{lem:M3locDensities} to~\eqref{eq:M3capT}, we obtain the desired bound on $N_{M_3\cap T}(m)$.
	\end{proof}


\subsection{Proof of Theorem~\ref{thm:MainThm} (1)}
\label{ss:ProofThm1.1}

	\begin{lemma}
		\label{lem:GHSineq}
		Let $t > 20$ be an integer such that $t \equiv 2$ or $11 \bmod 12$. Then
		\begin{equation}
			\label{eq:KeyInequality}
			10N_{M_2\cap T}(2t/3) + 15N_{M_3\cap T}(2t/3) < 4N_{M_1 \cap T}(2t/3).
		\end{equation}
	\end{lemma}

	\begin{proof}
		Elementary calculus shows that the inequality
		\[
			10\cdot \frac{3}{2}\cdot t(\log t + 1) + 15\cdot 9.0004 \cdot \frac{2}{3} \cdot t (\log t + 1) < 4\cdot 5.2488 \cdot \left(\frac{2}{3}\right)^{3/2}\cdot  t^{3/2}
		\]
holds for all integers $t > 8\,528$.  Applying Propositions~\ref{cor:M1Bound}, and~\ref{cor:M2Bound} we obtain the inequality~\eqref{eq:KeyInequality} for integers $t > 8\,528$ such that $t \equiv 2$ or $11 \bmod 12$. To improve the bound on $t$, we verify that for integers $1000 \leq t \leq 8\,528$ the inequality
		\[
			\sum_{a\mid t} \frac{1}{a} < 0.444\cdot (\log(t) + 1)
		\]
holds (this is an easy computer calculation). This allows us to improve the estimate~\eqref{eq:logisin} in this range of $t$, and hence we can sharpen Proposition~\ref{cor:M2Bound}, multiplying the right hand side of the inequalities given there by a factor of $0.444$. It is then natural to the consider instead the inequality
		\[
			10\cdot \frac{3}{2}\cdot 0.444\cdot t(\log t + 1) + 15\cdot 9.0004 \cdot 0.444\cdot \frac{2}{3} \cdot t (\log t + 1) < 4\cdot 5.2488 \cdot \left(\frac{2}{3}\right)^{3/2}\cdot  t^{3/2},
		\]
which holds provided $t > 1054$. Thus we improve the range for which~\eqref{eq:KeyInequality} holds to integers $t > 1054$. Finally, we explicitly calculate the initial terms of difference of theta series
		\[
			\begin{split}
				&4\Theta_{M_1 \cap T}(q) - 10\Theta_{M_2\cap T}(q) - 15\Theta_{M_3\cap T}(q) \\
				 = &\sum_{k = 0}^\infty \left(4N_{M_1 \cap T}\left(k/3\right) - 10N_{M_2\cap T}\left(k/3\right) - 15N_{M_3\cap T}\left(k/3\right)\right)q^{k/3}
			\end{split}
		\]
on a computer, up to $k = 2\cdot 1054$, and observe that the coefficients of $q^{k/3}$ with $46 \leq k \leq 2\,108$ and $k \equiv 4$ or $10 \bmod 12$ are positive. Set $k = 2t$. We conclude that if $t > 20$ is an integer congruent to $2$ or $11 \bmod 12$ then~\eqref{eq:KeyInequality} holds. The {\tt magma}~\cite{Magma} script {\tt ThetaFunctionComputations} verifying this computation is included in the {\tt arXiv} distribution of this article.
	\end{proof}

	\begin{theorem}
		\label{thm:Dis2mod6}
		Let $n > 18$ be an integer such that $n \notin \{20,21,25\}$, and set $d = 6n + 2$.  Then there exists an embedding of lattices $K_d^\perp(-1) \hookrightarrow L_{2,26}$ such that
		\begin{equation}
			\label{eq:rightnumberofroots}
			0 < \#R_{-2}\left(K_d^\perp(-1)\right) < 14.
		\end{equation}
For $n \in \{14,16,17,18,20,21,25\}$ there is an embedding such that $\#R_{-2}\left(K_d^\perp(-1)\right) = 14$.
	\end{theorem}

	\begin{proof}
		We apply Theorem~\ref{thm:reduction}. First, we determine the set of positive integers $n$ such that there exist coprime integers $\alpha$ and $\beta$, as well as an $\epsilon > 0$ all satisfying
		\begin{enumerate}
			\item the inequalities~\eqref{eq:notypeIIorIII}:
			\[
				\sqrt{(1 + \epsilon)n} < \alpha < \sqrt{5n/4},\qquad
				\sqrt{(1 + \epsilon)n} < \beta < \sqrt{5n/4}
			\]
			\item the congruence condition $t \equiv 2$ or $3 \bmod 4$, where $t := 3(\alpha\beta - n) -1$.
			\smallskip
			\item the inequality $t > 20$.
		\end{enumerate}
Condition (1) implies that $\epsilon < 1/4$. To guarantee coprimality of $\alpha$ and $\beta$ we impose the condition $\beta = \alpha + 1$. Suppose that
		\begin{equation}
			\label{eq:cond1}
			\sqrt{5n/4} - \sqrt{(1 + \epsilon)n} > 4.
		\end{equation}
Then there always exist two consecutive integers $\alpha$ and $\beta$ in the interval $\left(\sqrt{(1 + \epsilon)n},\sqrt{5n/4}\right)$ such that $(2)$ holds, regardless of the congruence class of $n$ modulo $4$. If, in addition, the inequality
		\begin{equation}
			\label{eq:cond2}
			3\epsilon n > 21
		\end{equation}
holds, then condition $(3)$ is satisfied, because
		\[
			t = 3\left(\alpha(\alpha + 1) - n\right) - 1 > 3(\alpha^2 -n) - 1 > 3\epsilon n - 1 > 20.
		\]
Simple optimization shows that if $\epsilon = 0.0058$, then for $n \geq 1207$ both~\eqref{eq:cond1} and~\eqref{eq:cond2} are satisfied, and this is the smallest integral value of $n$ that works. Thus, for $n \geq 1207$, we may apply Theorem~\ref{thm:reduction} (taking into account Lemma~\ref{lem:GHSineq}) to conclude.

		For $n < 1207$, we find an embedding $K_d^\perp(-1) \hookrightarrow L_{2,26}$ using the procedure outlined at the beginning of \S\ref{S:latticeEng}. Let $\langle a_1,a_2\rangle$ be a sublattice of $E_8(-1)$ of type $A_2(-1)$, and note that $A_2(-1)^\perp_{E_8(-1)} \isom E_6(-1)$.  We reduced the problem of finding an embedding satisfying~\eqref{eq:rightnumberofroots} to finding a vector $\ell \in U \oplus E_8(-1)$ of the form
		\[
			\ell = \alpha e + \beta f + v,\qquad\textrm{where } U = \langle e,f\rangle, v \in E_8(-1), \textrm{ and }\alpha, \beta \in \Z \textrm{ coprime},
		\]
of length $2n$, satisfying $(v,a_1) = 0$ and $(v,a_2) = 1$, in such a way that the set
		\[
			R_\ell = \left\{r \in U \oplus E_6(-1) : r^2 = -2, (r,\ell) = 0\right\}
		\]
is not empty and has $< 14$ elements (note that $R_\ell$ is finite because the orthogonal complement of $A_2(-1) + \langle\ell\rangle$ in $U\oplus E_8(-1)$ is a definite lattice).  As in \S\ref{ss:PrelimReds}, we decompose $v = v_1^\vee + v_2^\vee$, with $v_1^\vee \in E_6(-1)^\vee$ and $v_2^\vee = -(a_1+2a_2)/3 \in A_2(-1)^\vee$, so that $v_1^\vee$ has length $2(n - \alpha\beta) + 2/3$. The reason for this is, roughly speaking, that $E_6(-1)^\vee$ has many fewer vectors of length $2(n - \alpha\beta) + 2/3$ than $E_8(-1)$ has of length $2(n - \alpha\beta)$; this will speed up tremendously our computer searches for $v_1^\vee$, and hence $v$.

		By definition of $v_2^\vee$, the set 
		\begin{equation}
			\label{eq:countingusingv1star}
			\left\{r \in U \oplus E_6(-1) : r^2 = -2, (r,v_1^\vee) = 0\right\}
		\end{equation}
coincides with $R_\ell$.  The discussion following Remark~\ref{rem:types} shows that if $n < \alpha\beta < 2n$, then there are three types of vectors in $R_{\ell}$. Thus, for a \emph{fixed} $n < 1207$, to find a vector $v_1^\vee \in E_6^\vee(-1)$ such that the set~\eqref{eq:countingusingv1star} is not empty and has $< 14$ elements, we proceed as follows:
		\begin{enumerate}
			\item[(0)] Fix, once an for all, an embedding of $E_8(-1) \subseteq \Z^8$ (using, e.g.,~\cite[p.\ 121]{CS}), as well as a sublattice $\langle a_1,a_2\rangle = A_2(-1) \subseteq E_8(-1)$. Compute its orthogonal complement $E_6(-1) \subseteq E_8(-1)$ and the dual $E_6^\vee(-1)$ of this complement. Compute the set of roots $R(E_6(-1)) = \{r \in E_6(-1) : r^2 = -2\}$. Let $v_2^\vee = -(a_1 + 2a_2)/3 \in A_2^\vee(-1)$.
			\smallskip
			\item[(i)] Find positive integers $\alpha$ and $\beta$ such that $\gcd(\alpha,\beta) = 1$ and $n < \alpha\beta < 2n$. 
			\smallskip
			\item[(ii)] Compute the set $W := \{w \in E_6^\vee(-1) : w^2 = 2(n - \alpha\beta) + 2/3\}$.
			\smallskip
			\item[(iii)] Fix $w \in W$. Check that $w + v_2^\vee \in E_8(-1)$; if not, then replace $w$ with ${-w}$. Set $v = w + v_2^\vee$. Initialize a counter $m = 0$ for $N\left(K_d^\perp(-1)\right) = \#R_{-2}\left(K_d^\perp(-1)\right)/2$.
			\smallskip
			\item[(iv)] Fix $r \in R(E_6(-1))$.
			\smallskip
			\begin{itemize}
				\item If $(v,r) = 0$ then $\pm r \in R_{-2}\left(K_d^\perp(-1)\right)$ are Type I vectors. Add $1$ to $m$.
				\smallskip
				\item If $(v,r) \neq 0$, $(v,r) \equiv 0 \bmod \alpha$, and $(v,r) \not\equiv 0 \bmod \beta$, then 
				\begin{equation}
					\label{eq:TypeIIalg}
					\pm \left(-\frac{(v,r)}{\alpha}f + r \right) \in R_{-2}\left(K_d^\perp(-1)\right)
				\end{equation}
are Type II vectors. Add $1$ to $m$.
				\smallskip
				\item If $(v,r) \neq 0$, $(v,r) \not\equiv 0 \bmod \alpha$, and $(v,r) \equiv 0 \bmod \beta$, then 
				\begin{equation}
					\label{eq:TypeIIIalg}
					\pm \left(-\frac{(v,r)}{\beta}e + r \right) \in R_{-2}\left(K_d^\perp(-1)\right)
				\end{equation}
are Type III vectors. Add $1$ to $m$.
				\smallskip
				\item If $(v,r) \neq 0$, $(v,r) \equiv 0 \bmod \alpha$, and $(v,r) \equiv 0 \bmod \beta$, then~\eqref{eq:TypeIIalg} and~\eqref{eq:TypeIIIalg} are Type II and Type III vectors, respectively. Add $2$ to $m$.
			\end{itemize}
Carry out this procedure for each $r \in R(E_6(-1))$.
			\smallskip
			\item[(v)] If $0 < m \leq 7$, then we have succeeded in constructing the desired embedding $K_d^\perp(-1) \hookrightarrow L_{2,26}$. Otherwise, go back to step (iii).
		\end{enumerate}
		\smallskip
Note that if $m = 0$ or $m > 7$ for every $w \in W$, then we may go back to Step (i) and find a new pair of integers $\alpha$ and $\beta$ to work with. We implemented the above procedure in {\tt Magma}, resulting in a constructive proof of the existence of an embedding $K_d^\perp(-1) \hookrightarrow L_{2,26}$ with $0 < m < 7$ for $18 < n < 1207$ and $n \notin \{20,21,25\}$. Moreover, for $n \in \{14, 16, 17, 18, 20, 21, 25\}$, we found embeddings with $m = 7$, as desired. The {\tt Magma}~\cite{Magma} script {\tt Check6nplus2} that verifies these claims is included in the {\tt arXiv} distribution of this article. The script loads lists of explicit embeddings encoded in the form $(n,\alpha,\beta,v)$ for the relevant values of $n < 1207$.
	\end{proof}

	\begin{proof}[Proof of Theorem~\ref{thm:MainThm} (1)]
		By Theorem~\ref{thm:Dis2mod6}, if $n > 18$ and $n \notin\{20, 21, 25\}$, there exists some embedding of lattices $K_d^\perp(-1) \hookrightarrow L_{2, 26}$ such that $0 < N\left(K_d^\perp(-1)\right) < 7$. It follows from Theorem~\ref{thm:quasipullback} that the quasi-pullback $\Phi|_{K_d^\perp(-1)}$ of the Borcherds form is a cusp form of weight $<19$ for the stable orthogonal group of $K_d^\perp(-1)$. To apply Theorem~\ref{thm:LowWeightCusp}, we need to verify that this cusp form vanishes along ramification divisors of the projection $\pi : \mathcal D_{K_d^\perp(-1)} \rightarrow \mathcal F_{K_d^\perp(-1)}(\overline{\Gamma}^+_d)$. The proof of \cite[Proposition 8.13]{GHSHandbook} for $K3$ lattices applies to our case without any modification. Indeed, $K3$ lattices and our lattice $K_d^\perp(-1)$ have the same discriminant group, rank, and signature. These are enough to apply the proof.
	\end{proof}


\section{The case $d \equiv 0 \bmod 6$}
	\label{s:0mod6}


	In this section, we treat the case $d \equiv 0 \bmod6$. Here, the monodromy group $\overline{\Gamma}^+_d$, defined in~\S\ref{ss:ModuliSpecialFourfolds}, is larger than the stable orthogonal group $\widetilde{O}^+(K_d^\perp(-1))$, so this case requires special considerations. Many of our calculations are inspired by the ideas in \cite[\S4]{GHSDifferential}.


\subsection{Embeddings of $K_d^\perp(-1)$ into $L_{2,26}$ with desired properties}
	\label{ss:embeddingLwithprops}

	By Proposition~\ref{prop:Kdperp}, when $d\equiv 0 \bmod 6$, the lattice $K_d^\perp(-1)$ is isomorphic to
	\[
		A_2(-1) \oplus \langle 2n \rangle \oplus U \oplus E_8(-1) \oplus E_8(-1).
	\]
The monodromy group $\overline{\Gamma}^+_d$ is a finite index subgroup of $O^+(K_d^\perp(-1))$, and it contains the stable orthogonal group $\widetilde{O}^+(K_d^\perp(-1))$ as a subgroup of index $2$. In fact, the proof of \cite[Proposition 5.2.1]{HassettComp} shows that $\overline{\Gamma}^+_d$ is generated by the stable orthogonal group together with an extra involution $\sigma$ which acts trivially on $A_2(-1) \oplus E_8(-1) \oplus E_8(-1)$ and by negation on $\langle 2n \rangle \oplus U$. We would like to extend $\sigma$ to an involution $\tilde{\sigma} \in O^+(L_{2,26})$, and use the modularity of the Borcherds form $\Phi_{12}$ with respect to $O^+(L_{2,26})$ to show that the quasi-pullback
	\[
		\Phi|_{K_d^\perp(-1)} := \frac{\Phi_{12}(Z)}{\displaystyle \prod_{r \in R_{-2}(K_d^\perp(-1))/\{\pm 1\}}(Z,r)}\Bigg|_{\calD_{K_d^\perp(-1)}^\bullet}
	\]
is modular with respect to $\sigma$. This will allow us to apply the low-weight cusp form trick (Theorem~\ref{thm:LowWeightCusp}) to prove that $\overline{\Gamma}^+_d\backslash \calD_{K_d^\perp(-1)}$ is of general type.

	Recall that we embed $K_d^\perp(-1)$ into $L_{2, 26}$ by taking a summand $U \oplus E_8(-1) \oplus E_8(-1)$ identically into $L_{2,26}$, $A_2(-1)$ into $E_8(-1)$, and the generator $\ell$ of $\langle 2n \rangle$ into $A_2(-1)^\perp_{E_8(-1)} \oplus U \isom E_6(-1) \oplus U$ so that $\ell = \alpha e + \beta f + v$, where $e, f$ is a basis for $U$, and $v \in E_6(-1)$ is a vector of length $v^2 = 2(n - \alpha \beta)$. Since $\sigma$ acts on $\langle\ell\rangle \oplus U$ by negation, we want $\tilde{\sigma}$ to restrict to an involution $J$ of $E_6(-1)$ for which $v$ is a $(-1)$-eigenvector. To prove modularity of $\Phi|_{K_d^\perp(-1)}$, we need $\det(\tilde{\sigma}) = \det(\sigma)$; the former equals $\det(J)$ and the latter is $-1$. Furthermore, to construct $\tilde{\sigma}$, we need the involution $(J,\Id)$ of $E_6(-1)\oplus A_2(-1)$ to extend to an involution of $E_8(-1)$; by~\cite[Corollary~1.5.2]{Nikulin}, $J$ must belong to the stable orthogonal group $\widetilde{O}(E_6(-1))$.

	The group $\widetilde{O}(E_6(-1))$ is equal to the Weyl group $W(E_6(-1))$; see~\cite[p. 125]{CS}. Thus, $J$ can be viewed as an element of order $2$ in $W(E_6(-1))$ with determinant $-1$. Every such element is either a reflection associated to a root of $E_6(-1)$, or the composition of three reflections associated to mutually orthogonal roots. Taking $J$ to be a reflection associated to a root yields embeddings $K_d^{\perp}(-1)\hookrightarrow L_{2,26}$ for which $\#R_\ell \geq 30$; these embeddings are not useful for our purposes. Thus, we take $J$ to be a composition of three reflections associated to mutually orthogonal roots  $s_1$, $s_2$, and $s_3$ of $E_6(-1)$; we note that $J$ is unique up to conjugation. We consider $J$ as a linear transformation of $E_6(-1)\otimes \Q$, and denote the intersection of the $(+1)$ and $(-1)$-eigenspaces of $J$ with $E_6(-1)$ by $E_6^{J,+}$ and $E_6^{J,-}$, respectively.

	By~\cite[Corollary 1.5.2]{Nikulin}, the involution $(-\Id, -\Id, \Id, J, \Id, \Id)$ on
		\[
			U \oplus U \oplus A_2(-1) \oplus E_6(-1) \oplus E_8(-1) \oplus E_8(-1).
		\]
		extends to an involution $J'$ on $L_{2,26}$ and the restriction of $J'$ to $K_d^\perp(-1)$ is equal to $\sigma$. 
		
	\begin{lemma}
		\label{lem:modularforbiggergroup}
		Suppose that $v \in E_6^{J,-}$, and that for every $r \in R_\ell$, we have $J'(r) \neq -r$. Then the function $\Phi|_{K_d^\perp(-1)}$ is modular with respect to $\overline{\Gamma}^+_d$.
	\end{lemma}

	\begin{proof}
		Let us introduce coordinates $Z_1$, $z_2$, $Z_3 \in \mathcal D_{L_{2,26}}^\bullet$ corresponding to three direct summands of the sublattice of $L_{2,26}$ given by
		\[
			(U \oplus E_8(-1)^{\oplus2} \oplus A_2(-1) ) \oplus \langle \ell \rangle \oplus \langle \ell\rangle^\perp_{U \oplus E_6(-1)}.
		\]
Since $J'(r) \neq -r$ for all $r \in R_\ell$, we may choose a set of representatives for $R_\ell/\{\pm1\}$ that is stable under $J'$. Using this set of representatives, we see that
		\[
			\prod_{r \in R_l/\{\pm 1\}} (Z_1 + z_2 + Z_3, r)= \prod_{r \in R_l/\{\pm 1\}} (Z_3, r) = \prod_{r \in R_l/\{\pm 1\}} (J'(Z_3), J'(r)) = \prod_{r \in R_l/\{\pm 1\}} (J'(Z_3), r).
		\]
Now we proceed as follows:
		\begin{align*}
			\Phi|_{K_d^\perp(-1)}(\sigma (Z_1 + z_2)) 
			&= \frac{\Phi_{12}(\sigma(Z_1 + z_2) + Z_3)}{\prod_{r \in R_l/\{\pm 1\}} (\sigma(Z_1 + z_2) + Z_3, r)} \\
			&= \frac{\Phi_{12}(\sigma(Z_1 + z_2) + J'(Z_3))}{\prod_{r \in R_l/\{\pm 1\}} (\sigma(Z_1 + z_2) + J'(Z_3), r)} \\
			&= \frac{\Phi_{12}(\sigma(Z_1 + z_2) + J'(Z_3))}{\prod_{r \in R_l/\{\pm 1\}} (Z_3, r)} \Bigg|_{\mathcal D_{K_d^\perp(-1)}}\\
			&=  \frac{\Phi_{12}(J'(Z_1 + z_2 + Z_3))}{\prod_{r \in R_l/\{\pm 1\}} (Z_3, r)} \Bigg|_{\mathcal D_{K_d^\perp(-1)}}\\
			&=  \frac{\det(J')\Phi_{12}(Z_1 + z_2 + Z_3)}{\prod_{r \in R_l/\{\pm 1\}} (Z_3, r)} \Bigg|_{\mathcal D_{K_d^\perp(-1)}}\\
			&= -\Phi|_{K_d^\perp(-1)}(Z_1 + z_2) = \det(\sigma)\Phi|_{K_d^\perp(-1)}(Z_1 + z_2). \tag*{\qed}
		\end{align*}
		\hideqed
	\end{proof}

	To apply Lemma~\ref{lem:modularforbiggergroup} in the sequel, we shall need a sufficient condition to guarantee that $J'(r) \neq {-r}$ for all $r \in R_\ell$, at least when $d \gg 0$. 
	
	\begin{lemma}
		\label{lem:pexists}
		Let $\ell = \alpha e + \beta f + v$, with $v \in E_6^{J,-}$, i.e., $v$ is an integral linear combination of the roots $s_1$, $s_2$, and $s_3$.  Assume further that $v$ is not an integral combination of two roots. Suppose that $\alpha$ and $\beta$ satisfy the hypotheses of Lemma~\ref{lem:onlyTypeI}. Then $J'(r) \neq -r$ for all $r \in R_\ell$.
	\end{lemma}

	\begin{proof}
		By Lemma~\ref{lem:onlyTypeI}, we know that $R_\ell$ contains no Type II or III vectors, i.e., every $r \in R_\ell$ is a root in $E_6(-1)$. In this case $J'(r) = J(r)$, and $(r,\ell) = 0$ implies that $(r,v) = 0$. If $J(r) = -r$, then $r \in E_6^{J,-}$, and since $r$ is a root, we must have $r = \pm s_i$ for some $1\leq i \leq 3$. But then the equality $(r,v) = 0$ implies that $v$ must be an integral combination of at most two roots, because the only roots in $E_6^{J,-}$ are $\{\pm s_1, \pm s_2, \pm s_3\}$.
	\end{proof}

	To certify the primitivity of the embedding $K_d^\perp(-1) \hookrightarrow L_{2,26}$, we could assume that $(\alpha,\beta) = 1$. This works well for $n \gg 0$, but we need greater flexibility in our choice of $\alpha$ and $\beta$ when $n$ is small. The following lemma records a more flexible criterion to certify primitivity of the embedding we are looking for.
	
	\begin{lemma}
		\label{lem:primitivity}
		Let $\ell = \alpha e + \beta f + v$; assume that $3 \nmid (\alpha,\beta)$, and that $v$ is a primitive vector in $E_6(-1)$. Then the embedding 
		\[
			B_n \oplus U \oplus E_8(-1)^{\oplus 2} = K_d^\perp(-1) \hookrightarrow L_{2,26} = U^{\oplus 2} \oplus E_8(-1)^{\oplus 3}
		\] 
		given by mapping $U \oplus E_8(-1)^{\oplus 2}$ identically and taking the generators~\eqref{eq:Bn} of $B_n$ to $a_1$, $a_2$ and $\ell$, respectively, is primitive.
	\end{lemma}

	\begin{proof}
		The image of $B_n$ is contained in $U \oplus A_2(-1) \oplus E_6(-1)$, where $A_2(-1) = \langle a_1, a_2 \rangle$, and $E_6(-1) = A_2(-1)^\perp_{E_8(-1)}$.
		Let $u \in \langle a_1,a_2\rangle$ be a primitive vector, and let $c$ and $k$ be relatively prime integers. It suffices to show that the vector $cu + k\ell$ in the image of the embedding described is primitive. Suppose that $cu + k\ell$ is divisible by a positive integer $m$ in $U \oplus A_2(-1) \oplus E_6(-1)$. Then $m \mid k\cdot(\alpha,\beta)$ and $m$ divides $cu + kv$ in $A_2(-1)\oplus E_6(-1) \subseteq E_8(-1)$. Using~\cite[\S1.5]{Nikulin}, together with $D(A_2(-1)) \isom D(E_6(-1)) = \Z/3\Z$, we compute
		\[
			\frac{E_8(-1)}{A_2(-1)\oplus E_6(-1)} \isom \Z/3\Z.
		\]
		Hence $m \mid 3$, and the hypothesis that $3\nmid (\alpha,\beta)$ gives $m \mid k$. In turn this implies that $m \mid c$, so $m = 1$, because $c$ and $k$ are relatively prime.
	\end{proof}

	\begin{lemma}
		\label{lemma:lessthan14}
		Let $v \in E_6^{J,-}$ be a vector length $2(\alpha \beta -n)$. If $v$ is not an integral combination of two roots, then the number of roots in $E_6(-1)$ orthogonal to $v$ is less than 14.
	\end{lemma}

	\begin{proof}
		To avoid clutter, throughout this proof we shall suppress the ``$(-1)$'' in \emph{all} lattices, so for example, the lattice $E_6(-1)$ will be denoted $E_6$. 

		Suppose that there are at least $14$ roots in $E_6$ orthogonal to $v$.
The root lattices of rank at most $5$ having at least $14$ roots are 
		\[
			A_5, \quad D_5, \quad A_4\oplus A_1, \quad D_4 \oplus A_1, \quad A_3 \oplus A_2, \quad A_4, \quad D_4, \quad A_1 \oplus A_3.
		\]
For each one of these lattices $L$, we show that the containment $v \in L$ leads to a contradiction.

		\begin{enumerate}
			\item $A_3 \oplus A_2$ and $D_4 \oplus A_1$: These root lattices are not contained in $E_6$; see~\cite[Lemma 5.5]{GHSDifferential}.
			\smallskip
			\item $A_4$ and $A_1 \oplus A_4$: The involution $J$ preserves $A_4 \subseteq \langle v\rangle^\perp_{E_6}$, and the signature of the restriction $J|_{A_4}$ is either $(3,1)$ or $(2,2)$. Suppose first that the signature is $(3,1)$. Then the restriction $J|_{A_4}$ is a reflection with respect to a root; hence $v$ is orthogonal to a root which is a $(-1)$-eigenvector of $J$. This means that $v$ is a linear combination of two roots, a contradiction.

			Suppose next that the signature of $J|_{A_4}$ is $(2,2)$. Assume that the induced action of $J|_{A_4}$ on the discriminant group $D(A_4)$ is nontrivial. In this case, there are two mutually orthogonal roots $b_1$, and $b_2$ such that
			\[
				J|_{A_4} = -\sigma_{b_1} \sigma_{b_2},
			\]
where $\sigma_{b_1}$ and $\sigma_{b_2}$ are reflections with respect to $b_1$, $b_2$ respectively. Then $s_1$, $s_2$, $s_3$ are $(-1)$-eigenvectors of $J$ and $b_1$, $b_2$ are $(+1)$-eigenvectors. This implies that $s_1$, $s_2$, $s_3$, $b_1$, and $b_2$ are mutually orthogonal, whence $E_6$ contains $A_1^{\oplus 5}$. We claim that this is impossible. Assume to the contrary that $A_1^{\oplus 5}$ is contained in $E_6$. The lattice $M_4 = (A_1 \oplus A_1)_{E_6}^\perp$ contains $A_3$ as its root system, giving an inclusion of $A_1^{\oplus 3}$ in $A_3$. Taking the inclusions
			\[
				A_1^{\oplus 3} \subseteq A_3 \subseteq A_3^\vee \subseteq (A_1^{\oplus 3})^\vee
			\] 
modulo $A_1^{\oplus 3}$, we conclude that $D(A_1^{\oplus 3}) = (A_1^{\oplus 3})^\vee/A_1^{\oplus 3} \isom (\Z/2\Z)^3$ contains a subgroup $A_3^\vee/A_1^{\oplus 3}$ whose quotient by $A_3/A_1^{\oplus 3}$ is isomorphic to $D(A_3) \isom \Z/4\Z$, which is clearly impossible.

			Finally, assume that the induced action of $J|_{A_4}$ on the discriminant group $D(A_4)$ is trivial. Then there are two mutually orthogonal roots $b_1$, $b_2$ such that
			\[
				J|_{A_4} = \sigma_{b_1} \sigma_{b_2},
			\]
Since $b_1$ and $b_2$ are $(-1)$-eigenvectors for $J$, they are two of $s_1$, $s_2$, $s_3$. However, since $v$ is orthogonal to $b_1$ and $b_2$, we conclude that $v$ is a multiple of a root in $E_6^{J,-}$. This contradicts our assumption on $v$.
			\smallskip
			\item $A_5$: We claim that $\langle A_5 \rangle_{E_6}^{\perp} \isom A_1$; this contradicts our assumption on $v$. First, the discriminant group $D(A_5)$ has no nontrivial isotropic subgroups, and this implies that any embedding $A_5 \hookrightarrow E_6$ is primitive. Let $x$ be a primitive generator for $\langle A_5 \rangle_{E_6}^{\perp}$. Then
			\[
				\langle x \rangle \oplus A_5 \subset E_6 \subset E_6^\vee \subset \langle x \rangle^\vee \oplus A_5^\vee.
			\]
Primitivity of $A_5$ implies that each projection of the map
			\[
				H_{E_6}:=E_6/(\langle x \rangle \oplus A_5) \hookrightarrow D(\langle x \rangle) \oplus D(A_5)
			\]
is injective, so the order of $H_{E_6}$ has to divide the order of $D(A_5)$, which is $6$. On the other hand
			\[
				\det(E_6) \# H_{E_6}^2 = \det(\langle x \rangle) \det(A_5),
			\]
which implies that $x^2  = -( \# H_{E_6})^2/2$ (recall our lattices are negative definite). Since $x^2$ is an even integer, we have $\# H_{E_6} = 2$ or $6$. However, if $\# H_{E_6} = 6$, then $x^2 = {-18}$ and $D(\langle x \rangle) \isom \mathbb Z/18$. The group $H_{E_6}$ is isomorphic to the unique subgroup of $D(\langle x \rangle)$ of order 6. We denote this subgroup and its discriminant form by $H_q$ and $q$, respectively. Since $H_{E_6}$ is an isotropic subgroup in $D(\langle x \rangle) \oplus D(A_5)$, we must have $-q \isom q_{A_5}$. On the other hand, $q_{A_5}$ contains no nontrivial isotropic subgroups, but $q$ does. This contradiction shows that $\# H_{E_6} = 2$ and $x^2 = {-2}$, which implies that $v$ is a multiple of a root, contradicting our hypothesis. 
			\smallskip
			\item $D_5$: An explicit realization of the lattice $D_5$ is
			\[
				D_5 = \{(x_1, x_2,x_3,x_4,x_5) \in \mathbb Z^5 \mid x_1+x_2+x_3+x_4+x_5 \text{ is even} \};
			\]
considered as a lattice under the standard dot product (see~\cite[p. 117]{CS}). The orthogonal group $O(D_5)$ is generated by the group $S_5$ of all permutations of coordinates and the sign change of the last coordinate:
			\[
				O(D_5) =\left\langle S_5, \diag(1,1,1,1,-1).\right\rangle
			\]
The restriction $J|_{D_5}$ has the signature $(3,2)$, and there are three conjugacy classes of involutions in $O(D_5)$ with the signature $(3,2)$. Representatives of these classes are given by
			\[
				\diag(-1,-1,1,1,1), \quad 
				\begin{pmatrix}
					0 & 1 \\ 
					1 & 0
				\end{pmatrix}
				\oplus
				\diag(-1,1,1), \quad\textrm{and}\quad 
				\begin{pmatrix}
					0 & 1 \\ 
					1 & 0
				\end{pmatrix}
				\oplus
				\begin{pmatrix}
					0 & 1 \\ 
					1 & 0
				\end{pmatrix}
				\oplus 
				\begin{pmatrix}
					1
				\end{pmatrix}.
			\]
Each involution has a root as an $(-1)$-eigenvector, so one of the roots in $D_5$ is a $(-1)$-eigenvector for $J$. This root must be one of $s_1$, $s_2$ or $s_3$.  But then $v$ is orthogonal to this root, so $v$ cannot be an integral combination of $3$ roots.
			\smallskip
			\item $D_4$: the lattice $D_4$ is isomorphic to
			\[
				D_4 = \{(x_1, x_2,x_3,x_4) \in \mathbb Z^5 \mid x_1+x_2+x_3+x_4 \text{ is even} \},
			\]
considered as a lattice under the standard dot product. The orthogonal group $O(D_4)$ is generated by
			\[
				S_4, \quad\diag(1,1,1,-1),\quad\textrm{and}\quad
				\frac{1}{2}
				\begin{pmatrix}
					1 & 1 & 1 & 1\\
					1 &-1 & 1 &-1\\
					1 & 1 &-1 &-1\\
					1 &-1 &-1 &1
				\end{pmatrix}.
			\]
The restriction $J|_{D_4}$ has the signature $(2,2)$ or $(3,1)$, and there are four conjugacy classes of involutions with these signatures, with representatives given by
			\[
				\begin{pmatrix}
					 0 & -1 \\ 
					-1 & 0
				\end{pmatrix}
				\oplus
				\begin{pmatrix}
					0 & 1 \\ 
					1 & 0
				\end{pmatrix}, \ \ 
				\begin{pmatrix}
					0 & 1 \\
					1 & 0
				\end{pmatrix}
				\oplus\diag(1,-1), \ \ \diag(1,1,1,-1), \ \ \textrm{and}\ \ 
				\begin{pmatrix}
					0 & 1 \\
					1 & 0
				\end{pmatrix}
				\oplus\diag(1,1)
			\]
The first, second, and fourth cases are impossible since these involutions have a root as a $(-1)$-eigenvector, and we can argue as at the end of case 4). The third case is also impossible because if the restriction of $J$ is conjugate to the third involution, then there are two mutually orthogonal roots which are fixed by $J$. This means that $A_1^{\oplus 5}$ is contained in $E_6$, which is impossible, as we already argued in case 2) above.
			\smallskip
			\item $A_1\oplus A_3$: In this case, the involution $J$ preserves $A_1$ and $A_3$. If $J$ acts on $A_1$ as $-1$, then there is a root which is a $(-1)$-eigenvector of $J$. This contradicts our hypothesis, arguing as at the end of case 4). Thus $J$ fixes $A_1$, and the restriction $J|_{A_3}$ has signature $(1,2)$ or $(2,1)$. There are four conjugacy classes of involutions with these signatures in $O(A_3)$. One is obtained as $\sigma_b$, where $b$ is a root. This cannot be the case because then $b$ would be a $(-1)$-eigenvector of $J$, and we get a contradiction as before. Another is $-\sigma_b$, but this is impossible because if $r$ is a root orthogonal to $b$, then $r$ is a $(-1)$-eigenvector of $J$. The third involution is $\sigma_{b_1}\circ \sigma_{b_2}$ where $b_1$, $b_2$ are mutually orthogonal roots. However, this is again impossible. The final element we need to consider is $-\sigma_{b_1}\circ \sigma_{b_2}$. In this case, $b_1$, $b_2$ are $J$-fixed vectors, so $a_1, a_2, a_3, b_1, b_2$ are mutually orthogonal. This is impossible because $A_1^{\oplus 5}$ is not contained in $E_6$, as argued in case 2) above. $\hfill\qedhere$
		\end{enumerate}
	\end{proof}

	The existence of a vector $v$ as  in Lemmas~\ref{lem:pexists} and~\ref{lemma:lessthan14} is subject to some arithmetic constraints.  The lattice $E_6^{J,-} = \langle s_1, s_2, s_3\rangle \isom A_1(-1)^{\oplus 3}$ has $2(x_1^2 + x_2^2 + x_3^2)$ as its associated quadratic form.  Requiring $v$ to have length $2(\alpha\beta - n)$ and not to be a sum of two roots is thus equivalent to solving the diophantine equation
	\[
		x_1^2 + x_2^2 + x_3^2 = \alpha\beta - n,\qquad x_1x_2x_3 \neq 0.
	\]
By Legendre's theorem on sums of three squares, an integer not of the form $4^a(8b+7)$ is a sum of three squares.  The condition $x_1x_2x_3 \neq 0$ is more subtle, but can be dealt with nevertheless.

	\begin{theorem}
		\label{thm:Dis0mod6}
		Let $n > 18$ be an integer such that $n \notin \{20,22,23,25,30,32\}$, and set $d = 6n$.  Then there exists a primitive embedding of lattices 
		\[
			K_d^\perp(-1) \isom A_2(-1)\oplus\langle 2n\rangle \oplus U \oplus E_8(-1)^{\oplus 2} \hookrightarrow L_{2,26}
		\]
		such that
		\smallskip
		\begin{enumerate}
			\item $\ell \in U\oplus E_6^{J,-}$,
			\smallskip
			\item $0 < \#R_\ell < 14$, and
			\smallskip
			\item $J'(r) \neq -r$ for all $r \in R_\ell$.
		\end{enumerate}
		\smallskip
For $n \in \{17,23,25,32\}$ there exists a primitive embedding satisfying $(1)$ and $(3)$, such that $\#R_\ell = 14$.
	\end{theorem}

	\begin{proof}
		We construct $\ell = \alpha e + \beta f + v$, with $v \in E_6^{J,-}$ satisfying the hypotheses of Lemmas~\ref{lem:pexists} and~\ref{lemma:lessthan14}. To this end, we specify integers $\alpha$ and $\beta$ such that
		\smallskip
		\begin{enumerate}
			\item $(\alpha, \beta) = 1$ (to certify primitivity of our embedding),
			\smallskip
			\item the inequalities~\eqref{eq:notypeIIorIII} hold: 
			\[
				\sqrt{(1 + \epsilon)n} < \alpha < \sqrt{5n/4},\qquad
				\sqrt{(1 + \epsilon)n} < \beta < \sqrt{5n/4}
			\]
			\item $\alpha\beta -n$ is a sum of three nonzero squares (see the discussion immediately preceding the theorem).
		\end{enumerate}
		\smallskip
To guarantee the first condition, we take $\beta = \alpha+1$. An integer $m$ is a sum of three nonzero squares if $m$ is not of the form $4^a(8b+7)$ or $4^ak$ with 
	\[
		k \in \{1, 2, 5, 10, 13, 25, 37, 58, 85, 130\} \cup \{N\}
	\]
for some $N > 5\cdot 10^{10}$; see~\cite[p.\ 79, Theorem~7]{G85}. Moreover, it is not necessary to include $N$ in the above list if one assumes the Generalized Riemann Hypothesis. We want the integer $\alpha\beta - n = \alpha^2 + \alpha -n$ to avoid these numbers. Suppose that
		\begin{equation}
			\label{eq:widthinterval}
			\sqrt{(5/4)n} - \sqrt{(1+\epsilon)n} > 8.
		\end{equation}
Then, since $\alpha^2 + \alpha \equiv 0, 2, 4$, or $6 \bmod 8$, we can choose $\alpha$ in such a way that $\alpha^2 + \alpha -n \equiv 1$ or $2 \bmod 8$, guaranteeing that $\alpha^2 +\alpha -n$ is not of the form $4^a(8b+7)$. Since $\alpha^2+\alpha-n \not\equiv 0 \bmod 4$, to avoid numbers of the form $4^ak$, we further require that
		\begin{equation}
			\label{eq:epsn}
			\epsilon n > 130
		\end{equation}
so that $\alpha\beta -n > 130$. For the last exceptional number $N$, if we have $\alpha\beta -n = N$, then
		\[
			n > 4N > 20 \cdot 10^{10}.
		\]
Thus, if $\epsilon$ is sufficiently small, we will have
		\begin{equation}
			\label{eq:newwidthinterval}
			\sqrt{(5/4)n} - \sqrt{(1+\epsilon)n} > 16.
		\end{equation}
This way we can adjust $\alpha$ by $\pm8$ to avoid $N$. Solving the simple optimization problem~\eqref{eq:widthinterval} and~\eqref{eq:epsn}, we obtain $\epsilon = 0.02307$ and $n \geq 5636$. For this value of $\epsilon$, if $n > 20\cdot 10^{10}$, then~\eqref{eq:newwidthinterval} also holds, so we may adjust $\alpha$ to avoid $N$.

		For $n < 5636$ we find an embedding with the desired properties using a computer search. We describe the procedure briefly, as it is similar to the search described in the proof of Theorem~\ref{thm:Dis2mod6}.
		\begin{enumerate}
			\item[(0)] Fix, once and for all, an embedding of $E_6(-1) \subseteq \Z^8$ (using, e.g.,~\cite{CS}), as well as a sublattice $\langle s_1,s_2,s_3\rangle = E_6^{J,-} \subseteq E_8(-1)$. 
			\smallskip
			\item[(i)] Find positive integers $\alpha$ and $\beta$ such that $3\nmid (\alpha,\beta)$ and $n < \alpha\beta < 2n$. 
			\smallskip
			\item[(ii)] Compute the set $V := \{v \in E_6(-1) : v\textrm{ primitive and } v^2 = 2(n - \alpha\beta)\}$.
			\smallskip
			\item[(iii)] Fix $v \in V$. Initialize a counter $m = 0$ for $N\left(K_d^\perp(-1)\right) = \#R_{-2}\left(K_d^\perp(-1)\right)/2$.
			\smallskip
			\item[(iv)] Carry out step (iv) of the procedure described in the proof of Theorem~\ref{thm:Dis2mod6}. 			
			\smallskip
			\item[(v)] If $\alpha = \beta$ then add $1$ to $m$; see Lemma~\ref{lem:types}.
			\smallskip
			\item[(vi)] If $0 < m \leq 7$, then we have succeeded in constructing the desired embedding $K_d^\perp(-1) \hookrightarrow L_{2,26}$. Otherwise, go back to step (iii).		
		\end{enumerate}
If $m = 0$ or $m > 7$ for every $v \in V$, then we may go back to Step (i) and find a new pair of integers $\alpha$ and $\beta$ to work with. We obtain a constructive proof of the existence of an embedding $K_d^\perp(-1) \hookrightarrow L_{2,26}$ with $0 < m < 7$ for $20 < n < 5636$ and $n \notin  \{20,22,23,25,30,32\}$. Moreover, for $n \in \{17,23,25,32\}$, we found embeddings with $m = 7$, as desired. The {\tt Magma} \cite{Magma} script {\tt Check6n} that verifies these claims is included in the {\tt arXiv} distribution of this article. The script loads lists of explicit embeddings encoded in the form $(n,\alpha,\beta,v)$ for the relevant values of $n < 5636$.
	\end{proof}


\subsection{Ramification divisors}
	\label{ss:ramification}

	To complete the proof of Theorem~\ref{thm:MainThm}(2), it remains to check that the quasi-pullback of the Borcherds form vanishes along ramification divisors of the modular projection
	\[
		\pi \colon \mathcal D_{K_d^\perp(-1)} \rightarrow \overline{\Gamma}_d^+ \backslash \mathcal D_{K_d^\perp(-1)},
	\]
when $\#R_\ell < 14$. This is a delicate calculation, since the group $\overline{\Gamma}_d^+$ is larger than the stable orthogonal group $\widetilde{O}^+(K_d^\perp(-1))$.  To avoid further clutter, in this subsection we let 
	\[
		L := K_d^\perp(-1)\quad\textrm{and}\quad\Gamma := \overline{\Gamma}_d^+.
	\]

	For $r\in L$, we let $L_r$ denote the lattice $\langle r\rangle^\perp_L$. If $r$ is primitive, then set $\mathrm{div}(r) = (r, L)$; if in addition $(r,r) <0$, then define the \defi{rational quadratic divisor}
	\[
		\mathcal D_r := \{ [Z]\in \mathcal D_L \mid (Z,r) = 0 \} \isom \mathcal D_{\langle r\rangle_L^\perp}.
	\]
The vector $r$ is said to be \defi{reflective} if the reflection with respect to $r$,
	\[
		\sigma_r \colon v \mapsto v -\frac{2(v,r)}{(r,r)}r
	\]
is contained in $O(L)$; in this case we say $\calD_r$ is a \defi{reflective divisor}. The ramification divisor of the modular projection $\pi : \mathcal D_L \rightarrow \Gamma\backslash \mathcal D_L$ is the union of the reflective divisors with respect to $\Gamma$:
	\[
		{\rm Bdiv}(\pi) = \bigcup_{\substack{r \in L,\ \textrm{primitive} \\ (r,r) < 0,\ \sigma_r \in \Gamma \cup -\Gamma}}\mathcal D_r;
	\]
see~\cite[Corollary 2.13]{GHSInventiones} and~\cite[Equation (32)]{GHSHandbook}.

	\begin{lemma}
		\label{lem:dets}
		Let $r \in L$ be a primitive vector, and suppose that $-\sigma_r \in \widetilde{O}(L) $. Then the determinant of the lattice $\langle L_r\rangle_{L_{2,26}}^\perp$ divides $4$.
	\end{lemma}

	\begin{proof}
		The discriminant group $D(L)$ is isomorphic to $\Z/(d/3)\Z\oplus \Z/3\Z$; see~\cite[Proposition~3.2.5]{HassettComp}. On the other hand, by~\cite[Proposition~3.2]{GHSInventiones},  the existence of an $r \in L$ such that $-\sigma_r \in \widetilde{O}(L)$ implies that $D(L) \isom (\Z/2\Z)^m \oplus (\Z/D\Z)$ for some integers $m$ and $D$; in fact, $D$ is the exponent of $D(L)$. Comparing $2$-Sylow and $3$-Sylow subgroups, these facts imply that $D(L)$ is cyclic, and hence that $D = d$ and $m = 0$. Using~\cite[Proposition~3.2]{GHSInventiones}, we conclude that one of the following must hold:
		\begin{align*}
			r^2 &= \pm 2d \text{ and } \mathrm{div}(r) = d; \text{ or }\\
			r^2 &= \pm d \text{ and } \mathrm{div}(r) = d \text{ or } d/2.
		\end{align*}
Finally, since
		\[
			|\det L_r| = \frac{| r^2\cdot \det L |}{{\rm div}(r)^2}
		\]
(see~\cite[Lemma~7.2]{GHSHandbook}), we have $|\det L_r| = 1$, $2$, or $4$.
	\end{proof} 

	\begin{lemma}
		\label{lemma:caseofinvolution}
		Let $r\in L$ be a primitive vector. If $-\sigma_r \in \Gamma \setminus \widetilde{O}(L)^+$ then $(r,r) = \pm 6$ and $\mathrm{div}(r) = 3$ or $6$.
	\end{lemma}

	\begin{proof}
		The discussion at the beginning of~\S\ref{ss:embeddingLwithprops} shows that $-\sigma_r$ can written as $\sigma\cdot \tau$ for some $\tau \in \widetilde{O}(L)^+$, where $\sigma$ which acts trivially on $A_2(-1) \oplus E_8(-1) \oplus E_8(-1)$ and by negation on $\langle 2n \rangle \oplus U$. Since $-\sigma_r \notin  \widetilde{O}(L)^+$, its image under the composition
		\[
			\Gamma \subset O(L) \to O(D(L)) \isom O(D(\langle 2n \rangle) \oplus D(A_2(-1)))
		\]
is $(-\Id,\Id)$, and so $\sigma_r \equiv (\Id,-\Id)$ on $D(\langle 2n \rangle) \oplus D(A_2(-1))$. Let $r^\vee = r/\mathrm{div}(r)$. Then
		\[
			-r^\vee = \sigma_r(r^\vee) \equiv r^\vee \bmod A_2(-1)^{\vee} + L.
		\]
In particular, $6r^\vee \in L$ and therefore ${\rm div}(r) \mid 6$. Since
		\[
			\mathrm{div}(r) \mid (r,r) \mid 2\cdot \mathrm{div}(r).
		\]
we conclude that $\pm(r,r) = 2, 4, 6$ or $12$.

		If $(r,r) = \pm 2$, then one can show that $\sigma_r \in \widetilde{O}(L)$, contradicting our assumption. If $(r,r) = \pm 4$, then we have $\mathrm{div}(r) = 2$.  On the other hand, $\sigma_r$ acts by the identity on $D(\langle 2n \rangle)$, i.e., we have $\sigma_r(\ell^\vee) \equiv \ell^\vee \bmod L$. Hence
		\[
			\ell^\vee - \sigma_r(\ell^\vee) = \pm \frac{(r,\ell^\vee)}{2}r = \pm (r^\vee, \ell^\vee)r \in L.
		\]
This means that $(r^\vee, \ell^\vee) \in \Z$. Letting $x$ denote a generator of $D(A_2(-1))$, we have $3(r^\vee,x) = (r^\vee,3x) \in \Z$. But then $3r^\vee$ pairs integrally with two elements $\ell^\vee$ and $x$ that generate $D(L) \isom D(\langle 2n\rangle)\oplus D(A_2(-1))$. This implies that $\frac{3}{2}r = 3r^\vee \in (L^\vee)^\vee = L$, which contradicts primitivity of $r$. Thus $(r,r) \neq \pm 4$. One obtains a similar contradiction when $(r,r) = \pm 12$.
	\end{proof}

	\begin{lemma}
		\label{lemma:ramification}
		Let $L \hookrightarrow L_{2,26}$ be the embedding constructed in Theorem~\ref{thm:Dis0mod6} in the case where $\# R_l <14$. Then the quasi-pullback of the Borcherds form vanishes along ramification divisors of the modular projection $
\pi \colon \mathcal D_L \rightarrow \Gamma \backslash \mathcal D_L$.
	\end{lemma}

	\begin{proof}
		Let $r \in L$ be a primitive vector such that $(r,r) <0$ and $\sigma_r \in \Gamma \cup -\Gamma$.
We must show that the quasi-pullback $\Phi|_L$ vanishes along the reflective divisor $\mathcal D_r$.

		Suppose first that $\sigma_r \in \Gamma$. By Lemma~\ref{lem:modularforbiggergroup}, the function $\Phi|_L$ is a modular form with respect to $\Gamma$ with character $\det$ and $\det(\sigma_r)=-1$, it follows that $\Phi|_L$ vanishes along $\mathcal D_r$.

		Next, if $-\sigma_r \in \widetilde{O}(L)$, then by Lemma~\ref{lem:dets} the determinant of $\langle L_r \rangle_{L_{2,26}}^\perp$ is $1$, $2$, or $4$. 
Lattices of rank $8$ with these possible determinants contain one of the following root systems
		\[
			E_8, A_1 \oplus E_7, D_8;
		\]
see~\cite{CS88}. It follows that $\langle L \rangle_{L_{2,26}}^\perp$ contains at least 112 roots, and since $R_\ell < 14$, we conclude that $\Phi|_L$ vanishes along $\mathcal D_r$ with order at least 50.

		Suppose now that $-\sigma_r \in \Gamma\setminus \widetilde{O}(L)$. Then $(r,r) = -6$ and ${\rm div}(r) = 3$ or $6$, by Lemma~\ref{lemma:caseofinvolution}. The lattice $L$ contains two orthogonal hyperbolic planes (see~\cite[Corollary A.5]{Ma}), so that by the Eichler criterion~\cite[Lemma~3.5]{GHSCompositio}, the $\widetilde{O}(L)$-orbit of a primitive vector $v\in L$ is determined by two invariants: its length $(v,v)$ and the image of $v^\vee = v/\mathrm{div} (v)$ in the discriminant group $D(L)$. 

		Suppose that $\mathrm{div}(r)=3$. Then, as in the proof of Lemma~\ref{lemma:caseofinvolution}, we have $\sigma_r(\ell^\vee) - \ell^\vee \in L$, and hence $(r^\vee,\ell^\vee) \in \Z$. Viewing $r^\vee$ as an element of $D(\langle 2n\rangle)\oplus D(A_2(-1))$, we have $r^\vee = k\ell^\vee + x$ for some integer $k$ and some $x \in D(A_2(-1))$. The containment $(r^\vee,\ell^\vee) \in \Z$ shows that $k$ must be divisible by $2n$, and hence $r^\vee = x$. Thus a representative of the $\widetilde{O}(L)$-orbit of $r$ is $\pm (a_1-a_2)$ where $\{a_1, a_2\}$ is the standard basis for $A_2(-1)$. The generator $\ell$ for $\langle 2n \rangle$ is mapped into a direct sum of $U$ and $E_8(-1)$ in $L_{2,26}$, and the direct summand of $\ell$ in $E_8(-1)$ is a linear combination of three mutually orthogonal roots $s_1, s_2, s_3$ which are orthogonal to $A_2(-1) = \langle a_1, a_2\rangle$. Using {\tt Magma}~\cite{Magma}, it is easy to verify that $\langle a_1, a_2, s_1, s_2, s_3\rangle_{E_8(-1)}^\perp$ contains only two roots, but on the other hand, $\langle a_1+a_2, s_1, s_2, s_3\rangle_{E_8(-1)}^\perp$ contains more than two roots\footnote{See the script {\tt Lemma6-8check} in the {\tt arXiv} distribution of the article.}. This implies that there is a root in $L_{2,26}$ which is orthogonal to $L_r$ but not $L$. This implies that our quasi-pullback of the Borcherds form vanishes along $\mathcal D_r$. 

		Assume that $\mathrm{div}(r) = 6$. In this case, we have $(r,r) = -6$, and
		\[
			r^\vee = (\pm 1, n) \in D(L) \isom \mathbb Z/3\mathbb Z \oplus\mathbb Z / 2n\mathbb Z.
		\]
First let us assume that $n \equiv 1 \bmod 4$. Then a representative of $r$ has the following form
		\[
			\pm2(a_1-a_2) + 3\ell + (6e +6\beta f) \in A_2(-1) \oplus \langle 2n \rangle \oplus U.
		\]
As we discussed before, there is a root in $L_{2,26}$ which is orthogonal to $L_r$, but not $L$. Thus we can conclude that our quasi-pullback of the Borcherds form vanishes along $\mathcal D_r$. 

		Suppose that $n\not\equiv 1 \bmod 4$. Since $\mathrm{div}(r) = 6$ and $\sigma_r \equiv (-I, I)$ on $D(L)$, $r$ has the form of
		\[
			r = 2v + 3k\ell + (6\alpha e + 6 \beta f) + 6w_1 + 6w_2 \in A_2(-1) \oplus \langle 2n \rangle \oplus U \oplus E_8(-1) \oplus E_8(-1),
		\]
where $k$ is odd. Thus we have
		\[
			-6 = r^2 \equiv 4v^2 +18k^2n\mod 72.
		\]
If $n$ is even, then arrive at a contradiction because the left hand side is not divisible by $4$. If $n \equiv 3 \bmod 4$, then
		\[
			-6 \equiv 4v^2 + 54k^2 \equiv 4v^2 + 54\bmod 72,
		\]
which implies that 
		\[
			v^2 \equiv 3\bmod 18.
		\]
This is a contradiction since $v^2$ is even.
	\end{proof}

	\begin{proof}[Proof of Theorem~\ref{thm:MainThm}(2)]
		Apply Theorems~\ref{thm:LowWeightCusp} and~\ref{thm:quasipullback}, using Theorem~\ref{thm:Dis0mod6} and Lemma~\ref{lemma:ramification}.
	\end{proof}


\end{document}